\documentclass[12pt,a4paper]{amsart}
\usepackage{algorithm}
\usepackage{algorithmic}
\usepackage{amssymb}
\usepackage{eucal}
\usepackage{graphicx}
\usepackage{amsmath}
\usepackage{amscd}
\usepackage[all]{xy}           
\usepackage{amsfonts,latexsym}
\usepackage{xspace}
\usepackage[T1]{fontenc}
\usepackage{tikz-cd}
\usepackage{txfonts}

\usepackage{epsfig}
\usepackage{float}
\usepackage{color}
\usepackage{fancybox}
\usepackage{colordvi}
\usepackage{multicol}
\usepackage{colordvi}
\ifpdf
\usepackage[colorlinks,final,backref=page,hyperindex]{hyperref}
\else
\usepackage[colorlinks,final,backref=page,hyperindex]{hyperref}
\fi
\usepackage[active]{srcltx} 
\usepackage{mathrsfs} 

\newtheorem{thm}{Theorem}[section]
\newtheorem*{theorem*}{Theorem}

\newtheorem{prop}[thm]{Proposition}
\newtheorem{lm}[thm]{Lemma}

\newtheorem{coro}[thm]{Corollary}

\makeatletter

\newcommand{\nc}{\newcommand}

\nc{\delete}[1]{{}}

    \nc{\mlabel}[1]{\label{#1}}  
    \nc{\mcite}[1]{\cite{#1}}  
    \nc{\mref}[1]{\ref{#1}}  
    \nc{\meqref}[1]{\eqref{#1}}  
    \nc{\mbibitem}[1]{\bibitem{#1}} 

\delete{
    \nc{\mlabel}[1]{\label{#1} {{\small\tt{{\ }\ (#1)}}}}                
    \nc{\mcite}[1]{\cite{#1}{{\small\tt{{\ }(#1)}}}}  
    \nc{\mref}[1]{\ref{#1}{{\small\tt{{\ }(#1)}}}}  
    \nc{\meqref}[1]{\eqref{#1}{{\small\tt{{\ }(#1)}}}}  
    \nc{\mbibitem}[1]{\bibitem[\bf #1]{#1}} 
}

\newcommand*\bigcdot{\mathpalette\bigcdot@{.5}}
\newcommand*\bigcdot@[2]{\mathbin{\vcenter{\hbox{\scalebox{#2}{$\m@th#1\bullet$}}}}}
\makeatother

\newtheorem{innercustomgeneric}{\customgenericname}
\providecommand{\customgenericname}{}
\newcommand{\newcustomtheorem}[2]{%
    \newenvironment{#1}[1]
    {%
        \renewcommand\customgenericname{#2}%
        \renewcommand\theinnercustomgeneric{##1}%
        \innercustomgeneric
    }
    {\endinnercustomgeneric}
}
\newcustomtheorem{customthm}{Theorem}

\theoremstyle{definition}
\newtheorem{example}[thm]{Example}
\newtheorem{df}[thm]{Definition}
\newtheorem{remark}[thm]{Remark}

\nc{\name}[1]{{\bf #1}}

\nc{\mhat}[1]{\,\widehat{#1}}

\allowdisplaybreaks

\newcommand{\N}{\mathbb{N}}
\newcommand{\Z}{\mathbb{Z}}
\newcommand{\Q}{\mathbb{Q}}
\newcommand{\R}{\mathbb{R}}

\nc{\Res}{\mathrm{Res}}

\def \ra {\rightarrow}

\def \C {\mathbb{C}}

\def\la{\lambda}

\def \al{\alpha}
\def \si{\sigma}
\def \om{\omega}
\def \ga {\gamma}

\def \b {\beta}
\def \op {\oplus}

\def \ssq{\subseteq}

\def \vac {\mathbf{1}}

\def \g {\mathfrak{g}}
\def \h {\mathfrak{h}}

\def \Hom {\mathrm{Hom}}

\def \End {\mathrm{End}}

\def\Id{\mathrm{Id}}
\def \wt {\mathrm{wt}}

\def \O {\mathcal{O}}

\def\gl{\mathfrak{gl}}
\def\T{\mathcal{T}}
\def\R{\mathcal{R}}
\def\o{\otimes}
\def\Y{\mathcal{Y}}
\def\SD{\mathrm{SD}}

\def\<{\langle}
\def\>{\rangle}

\newcommand\numberthis{\addtocounter{equation}{1}\tag{\theequation}}

\begin{document}

\title[Yang-Baxter equation for VOAs]{Classical Yang-Baxter equation for vertex operator algebras and its operator forms}

\author{Chengming Bai}
\address{Chern Institute of Mathematics \& LPMC, Nankai University, Tianjin 300071, China}
\email{baicm@nankai.edu.cn}

\author{Li Guo}
\address{Department of Mathematics and Computer Science, Rutgers University, Newark, NJ 07102, USA}
\email{liguo@rutgers.edu}

\author{Jianqi Liu}
\address{Department of Mathematics, University of California, Santa Cruz, CA 95064, USA}
\email{jliu230@ucsc.edu}

\date{\today}

\begin{abstract}
In this paper we introduce an analog of the (classical) Yang-Baxter equation (CYBE) for vertex operator algebras (VOAs) in its tensor form, called the vertex operator Yang-Baxter equation (VOYBE). When specialized to level one of a vertex operator algebra, the VOYBE reduces to the CYBE for Lie algebras. To give an operator form of the VOYBE, we also introduce the notion of relative Rota-Baxter operators (RBOs) as the VOA analog of relative RBOs (classically called $\mathcal{O}$-operators) for Lie algebras. It is shown that skewsymmetric solutions $r$ to the VOYBE in a VOA $U$ are characterized by the condition that their corresponding linear maps $T_r:U'\to U$ from the graded dual $U'$ of $U$ are relative RBOs. On the other hand, strong relative RBOs on a VOA $V$ associated to an ordinary $V$-module $W$ are characterized by the condition that their antisymmetrizers are solutions to the $0$-VOYBE in the semidirect product VOA $V\rtimes W'$. Specializing to the first level of a VOA, these relations between the solutions of the VOYBE and the relative RBOs for VOAs recover the classical relations between the solutions of the CYBE and the relative RBOs for Lie algebras.
\end{abstract}

\subjclass[2010]{
17B69,	
17B38, 
17B10, 
81R10  
17B68, 
17B65,  
81R12  
}

\keywords{Vertex algebra, vertex operator algebra, Lie algebra, classical Yang-Baxter equation, vertex operator Yang-Baxter equation, Rota-Baxter operator, $\mathcal{O}$-operator, conformal algebra}

\maketitle

\vspace{-1.2cm}

\tableofcontents

\vspace{-.5cm}

\allowdisplaybreaks

\section{Introduction}

This paper gives a systematic study of the analog of the classical Yang-Baxter equation for vertex operator algebras, in both its tensor form and operator form. The relation to the classical Yang-Baxter equation on Lie algebras is established.

\subsection{Classical Yang-Baxter equations and Rota-Baxter operators}
Classical Yang-Baxter equations for various algebraic structures have played a major role in mathematics and physics.
\subsubsection{Classical Yang-Baxter equation}
The classical Yang-Baxter equation (CYBE) is the semi-classical limit of the quantum Yang-Baxter equation, named after the physicists C.N. Yang and R. Baxter~\mcite{BaR,Yang} in statistics mechanics. A more general form was introduced by L.D. Faddeev, E.K. Sklyanin, L.A. Takhtazhyan~\mcite{Fa} as a basis for their quantum inverse scattering method. The independent importance of the CYBE lies in its crucial role in broad areas including symplectic geometry, integrable systems, quantum groups and quantum field theory (see~\mcite{BGN, Dr,Ji,  KS} and the references therein). Its solution has inspired the remarkable works of Belavin, Drinfeld and others, and the important notions of Lie bialgebras and Manin triples~\mcite{BD}.

Let $\g$ be a Lie algebra and
$r=\sum_{i}x_{i}\otimes y_{i}$ for $x_i, y_i\in \g$.
The tensor form   of the CYBE  is
\begin{equation}\mlabel{eq:1}
	[[r,r]]:=[r_{12}, r_{13}]+[r_{12}, r_{23}]+[r_{13}, r_{23}]=0,
\end{equation}
where
\begin{equation}
	r_{12}:=\sum_{i}x_{i}\otimes y_{i}\otimes 1,\ \
	r_{13}:=\sum_{i}x_{i}\otimes 1\otimes y_{i},\ \
	r_{23}:=\sum_{i}1\otimes x_{i}\otimes y_{i}.
\end{equation}
Here $\g$ is embedded in its universal
enveloping algebra with unit $1$.
A solution of the CYBE is  called a classical
$r$-matrix (or simply an $r$-matrix).

\subsubsection{Relative Rota-Baxter operators as the operator form}
Since the early stage of the study, it has been found necessary to expand the tensor form of the CYBE in~\meqref{eq:1} to operator
equations. First Semenov-Tian-Shansky~\mcite{S} showed that skewsymmetric $r$-matrices in certain Lie algebras can be characterized as
linear operators satisfying a (Rota-Baxter) operator identity. The modified Yang-Baxter equation was also defined by another operator identity. Then the notion of an $\mathcal O$-operator (later called a relative Rota-Baxter operator) of a Lie algebra was introduced by
Kupershmidt~\mcite{K1} as a natural generalization of the CYBE,
and a skewsymmetric $r$-matrix is interpreted as a Rota-Baxter operator associated to the coadjoint representation. Conversely, it was shown in~\mcite{B2}
that any relative Rota-Baxter operator gives an $r$-matrix  in a semidirect product Lie algebra.
Furthermore,  pre-Lie algebras, in addition to their independent
interests in deformation theory, geometry, combinatorics, quantum field
theory and operads~\mcite{Bu,Man}, naturally give rise to relative Rota-Baxter operators and hence provide $r$-matrices~\mcite{B2,TBGS}.
Interestingly, the Rota-Baxter identity on associative algebras has appeared in the probability study of G.~Baxter in 1960~\mcite{BaG}, and recently appeared in the algebraic approach of Connes and Kreimer to renormalization of quantum field theory~\mcite{CK} as a fundamental structure. See~\mcite{G}. The associative analog of the modified Yang-Baxter equation appeared even earlier in 1951~\mcite{Tr}.

\subsubsection{Classical Yang-Baxter equations for other structures}
Inspired by the importance of the CYBE in the theory and applications of Lie algebras, analogs of the CYBE have been studied for various other algebraic structures, including associative algebras, Leibniz algebras, Poisson algebras, pre-Lie algebras, Novikov algebras, Rota-Baxter Lie and associative algebras, Hom-Lie algebras, endo Lie algebras, $3$-Lie algebras and Lie 2-algebras~\mcite{A1,BGS19,BGS20,BSZ13,HB20,HBG,NB,TS22}. 

\subsection{Vertex operator algebras}
Vertex algebras were introduced by Borcherds~\mcite{Bo} in 1986 to give a rigorous mathematical description of the conformal field theory, and to find an algebraic structure to generalize the basic representation of affine Lie algebras associated with A-D-E type root lattice \cite{FK} to the Leech lattice.
A vertex operator algebra (VOA) is a graded vertex algebra with a distinguished element (called the Virasoro element) that encodes the action of conformal transformations on the algebra.
Introduced by Frenkel, Lepowsky, and Meurman~\mcite{FLM} in 1988 to address the problems of the ``monstrous moonshine'' phenomena for the monster finite simple group and $j$-invariant, VOAs have been the subject of intensive research, and have found applications in diverse areas of mathematics and physics, including representation theory, string theory, combinatorics, and geometry~\mcite{Bo,FLM,DL,Hu97,K,LL}.

Vertex (operator) algebras are intimately related to Lie algebras. The Jacobi identity of Lie algebras motivates the Jacobi identity in a vertex algebra. It was known~\mcite{Bo} that for a vertex algebra $V$, all components (or modes) of vertex operators give rise to a Lie algebra $\mathfrak{g}(V)$. On the other hand, the representations of affine Lie algebras are used to construct vertex operator algebras~\cite{FLM,FZ}. In turn, these VOAs are determined by their first-level Lie algebras $\g=V_1$.

In addition, Frenkel and Zhu~\mcite{FZ,Z} introduced the associative algebra $A(V)$ associated to a VOA $V$ such
that there is a one-to-one correspondence between irreducible admissible $V$-modules and
irreducible $A(V)$-modules. This construction was extended in~\mcite{DLM} to a sequence of associative algebras constructed from a VOA. Recently, Huang constructed an associative algebra $A^\infty(V)$ which contains these previous algebras as subalgebras~\mcite{Hu1}. VOAs are also closely related to (differential) commutative algebras and perm algebras~\mcite{LTW}. 

\vspace{-.1cm}

\subsection{VOAs and CYBEs}
\vspace{-.1cm}

\subsubsection{Challenges and motivations}
As noted above, variations of the CYBE have been defined for many algebraic structures and have played major roles in their studies, such as their bialgebra theories, deformations and cohomologies. The VOA, defined by packaging together infinitely many binary operations, is much more complicated  in its definition than the traditional algebraic structures. In addition, unlike some other algebraic structures with infinitely many operators such as $L_\infty$-algebras, the VOA does not fit into the usual algebraic structures as defined by varieties. Indeed, as already noted by Borcherds~\mcite{Bo}, free vertex algebras do not exists in general since the vertex algebras do not form a variety of algebras in the sense of universal algebra~\mcite{Ro}.

Thus it is challenging to define an analog of the CYBE for VOAs. Indeed, so far such an analog is established only in a very special case for the operator form, namely for Rota-Baxter operators~\mcite{X1,X2}. However, there was no tensor form analog of the CYBE for VOAs.

In our search of a suitable notion of the CYBE for VOAs which we call the vertex operator Yang-Baxter equation (VOYBE), we have in mind not only a more general operator form of the equation, but also a tensor form that is comparable with the tensor form of the CYBE.
Furthermore, due to the close relationship of VOAs with several common algebraic structures, including Lie algebras, associative algebras, differential algebras and perm algebras, it should be beneficial to compare the VOYBE with the analogs of CYBE that have been defined for all these algebraic structures.
In addition, such VOYBE should also contribute to establishing a bialgebra theory of VOAs.

On the other hand, the original (quantum) version of the Yang-Baxter equation in \mcite{BaR,Yang} was used to introduce the notion of deformed chiral algebras or quantum VOAs, see \mcite{FR,EK}, and the quantum Yang-Baxter equation is part of the axioms of these algebras.

In summary, the analog of the CYBE for VOAs is in the center of several interesting questions on VOAs and their connections. This paper is a first step in exploring these connections.

\subsubsection{Our approach}
In this paper, we introduce both a tensor form and an operator form of the classical Yang-Baxter equations for VOAs, so that they are compatible in a way that generalizes the compatibility of the tensor form and operator form of the CYBE, when a VOA $V$ is specialized to the Lie algebra on the first level $V_1$. In particular, for each index $m$ in the vertex operator $Y(a,z)=\sum_{m\in \Z} a_m z^{-m-1}$, we give the $m$-vertex operator Yang-Baxter equation ($m$-VOYBE) 
\vspace{-.1cm}
\begin{equation}\mlabel{3}
	r_{12}\cdot_m r_{13}- r_{23}\cdot'_m r_{12}+r_{13}\cdot'^{\mathrm{op}}_m r_{23}=0.
\end{equation}
In this equation, $r$ is in the completion $V\mhat{\o}V$ of the tensor product $V\o V$ with respect to the natural filtration of $V\o V$, and the three products $\cdot_m, \cdot'_m$, and $\cdot'^{\mathrm{op}}_m$ are constructed from the vertex operator $Y$.

For the operator form of the VOYBE, we also introduce the notion of relative Rota-Baxter operators (RBOs)  for the VOAs, as a generalization of both the relative RBOs for the Lie algebras and the operator form generalization of the $R$-matrices for VOAs given in \mcite{X1}. A relative Rota-Baxter operator of the VOA $(V,Y_V,\vac,\om)$ associated with an ordinary $V$-module $(W,Y_W)$ is a linear map $T:W\ra V$, satisfying the axiom
\vspace{-.1cm}
\begin{equation}\mlabel{4}
Y_V(T(u),z)T(v)=T(Y_W(T(u),z)v)+T(Y_{WV}^W(u,z)T(v)), \quad u, v\in W.
\end{equation}

We will show that the correspondence between skewsymmetric solutions of the tensor form of CYBE and the relative RBOs of the Lie algebra $\g$ associated with the coadjoint representation given in \mcite{K1} can be naturally generalized to the case of VOA, and the correspondence in the VOA case recovers the correspondence of the Lie algebra case in \mcite{K1}.

On the other hand, the first author proved in \mcite{B2} that  relative RBOs of a Lie algebra $\g$ associated with a module $V$ give rise to skewsymmetric solutions of the tensor form of the CYBE in the semidirect product Lie algebra $\g\rtimes V^\ast$. With our definition of the VOYBE and relative Rota-Baxter operators for VOAs, a similar result holds, except the relative Rota-Baxter operator needs to satisfy some additional compatibility properties with certain intertwining operators of VOAs.

\vspace{-.1cm}
\subsection{Outline of the paper} The paper is organized as follows.
In Section~\mref{s:voybe}, we apply representations of VOAs and complete tensor products of graded vector spaces to define the VOYBE (Definition~\mref{df:VOYBE}). For the operator form of the VOYBE, we further introduce the notion of relative Rota-Baxter operators for a VOA associated to a representation of the VOA (Definition~\mref{df3.1}). As a first justification of the operator form, we show that skewsymmetric solutions of the VOYBE corresponds to relative Rota-Baxter operators of the VOA associated to its adjoint representation (Theorem~\mref{thm:main1}).

The second justification for relative RBOs as the operator form of  the VOYBE is given in Section~\mref{s:ofvoybe}, where we prove that any relative Rota-Baxter operator of a VOA associated to a representation $T:W\ra V$ gives rise to a solution of the VOYBE in the semidirect product VOA $V\rtimes W'$, where $W'$ is the contragredient module of $W$ (Theorem~\mref{thm:main2}).

In Section~\mref{s:cybe}, we establish the connection between the VOYBE and the CYBE of Lie algebras, by considering the reduction process of a VOA to its Lie algebra at the first degree. We show that in this reduction process, a solution to the VOYBE in this VOA is reduced to a solution of the CYBE in the reduced Lie algebra, a relative RBO for this VOA is reduced to a relative RBO (classically called an $\mathcal{O}$-operator) for the corresponding Lie algebra. Furthermore, the relations between solutions of the VOYBE and relative RBOs for the VOAs established in Theorems~\mref{thm:main1} and \mref{thm:main2} reduce to the classical relations between the solutions of the CYBE on Lie algebras and the relative RBOs for Lie algebras (see diagrams \meqref{eq:diag1} and \meqref{eq:diag2}).

\smallskip
\noindent
{\bf Conventions.} Throughout this paper, we take the field $\C$ of complex numbers to be the base field of vector spaces, linear maps and tensor products. Let $\N$ denote the set of nonnegative integers.

\section{Vertex operator Yang-Baxter equation}
\mlabel{s:voybe}

In this section, we first recall the needed background on VOAs and completed tensor products. We then give the notion of the VOYBE, as the VOA analog of the (tensor form of the) CYBE. This is followed by the notion of relative Rota-Baxter operators which serves as the operator form of the VOYBE.

\subsection{Vertex operator algebras and their representations}
\mlabel{ss:voa}
We recall basic notions of vertex algebras and VOAs. For details, see for example~\mcite{Bo,DL,FHL,FLM,Hu97,K,LL}.

\begin{df}\mlabel{df:va}
	A \name{vertex algebra (VA)} is a triple $(V, Y, {\bf 1} )$ consisting of a vector space $V$,
	a linear map
	\begin{align*}
	Y:& V \to (\mbox{End}\,V)[[z,z^{-1}]] ,\quad 
	 a\mapsto Y(a,z)=\sum_{n\in{\Z}}a_nz^{-n-1}, \ \ \ \  \text{ where } a_n\in
	\mbox{End}\,V,\nonumber
	\end{align*}
which is called the \name{vertex operator} or the \name{state-field correspondence},	and a distinguished element ${\bf 1} \in V$ called the \name{vacuum vector}, satisfying the following conditions.
	\begin{enumerate}
		\item  (Truncation property) For $a,b\in V$, $ a_nb=0$ when $n\gg 0$.
		\item  (Vacuum property) $Y(\vac,z)=\Id_{V}$.
		\item  (Creation property) For $a\in V$, $Y(a,z){\bf 1}\in V[[z]]$ and $ \lim\limits_{z\to 0}Y(a,z){\bf 1}=a$.
		\item (The Jacobi identity) For $a,b\in V$,
		\begin{align*}& \displaystyle{z^{-1}_0\delta\left(\frac{z_1-z_2}{z_0}\right)
			 Y(a,z_1)Y(b,z_2)-z^{-1}_0\delta\left(\frac{-z_2+z_1}{z_0}\right)
			Y(b,z_2)Y(a,z_1)}
		\displaystyle{=z_2^{-1}\delta
			 \left(\frac{z_1-z_0}{z_2}\right)
			Y(Y(a,z_0)b,z_2)}.
		\end{align*}
Here $\delta(x):=\sum_{n\in \Z} x^n$ is the  formal delta function. 
	\end{enumerate}
\end{df}
With a distinguished ``conformal element'', we have the following enhancement of VAs~\mcite{FLM}.  
\begin{df}\mlabel{df2.2}
	A \name{vertex operator algebra (VOA)} is a quadruple $(V, Y, \vac, \omega)$, where
\begin{itemize}
\item $(V,Y(\cdot, z), \vac)$ is a $\Z$-graded vertex algebra: $V=\bigoplus_{n \in \Z}V_n,$
	such that ${\bf 1} \in V_0$, $\mathrm{dim}V_n < \infty$ for each $n\in \Z$, and $V_n=0$ for $n$ sufficiently small;
\item ${\omega} \in V_2$ is a distinguished element, called the \name{Virasoro element}. Write $Y(\omega,z)=\sum_{n\in \Z} a_n z^{-n-1}=\sum_{n\in\Z}L(n)z^{-n-2}$, that is, $L(n):=\omega_{n+1}$ for $n\in \Z$.
\end{itemize}
Together, they fulfill the following additional conditions.
	
	\begin{enumerate}
		\item[(v)] (The Virasoro relation)
		 $$[L(m),L(n)]=(m-n)L(m+n)+\frac{1}{12}(m^3-m)\delta_{m+n,0}c,$$
		where $c\in \C$ is called the \name{central charge} (or \name{rank}) of $V$.
		\item[(vi)] ($L(-1)$-derivation property) $L(-1)a=a_{-2}\vac$, and  $$\frac{d}{dz}Y(v,z)=Y(L(-1)v,z)=[L(-1),Y(a,z)].$$
		\item[(vii)] ($L(0)$-eigenspace property) $L(0)v=nv$, for all $v\in V_{n}$ and $n\in \Z$.
	\end{enumerate}

A VOA $V$ is said to be of \name{conformal field type (CFT-type)}, if $V=V_0\op V_+$, where $V_0=\C\vac$ and $V_+=\bigoplus_{n=1}^\infty V_n$ as a graded vector space.
\end{df}
For a homogeneous element $a\in V_n$, we denote $\wt(a):=n$.

\begin{df}\mlabel{df2.13}
	Let $(V,Y,\vac,\om)$ be a vertex operator algebra, a \name{weak $V$-module} $(W,Y_{W})$ is a vector space $W$ equipped with a linear map
$$	Y_{W}:V\ra (\End W)[[z,z^{-1}]],\quad
	a\mapsto Y_{W}(a,z)=\sum_{n\in \Z} a_{n} z^{-n-1},\quad  \text{where}\ a_n\in \End(W),
$$	
satisfying the following axioms. 
	\begin{enumerate}
		\item (Truncation property) For any $a\in V$ and $v\in W$, $a_{n}v=0$ for $n\gg 0$.
		\item (Vacuum property) $Y_{W}(\vac,z)=\Id_{W}$.
		\item (The Jacobi identity) For any $a,b\in V$, and $u\in W$
		\begin{align*}
		 &z_{0}^{-1}\delta\left(\frac{z_{1}-z_{2}}{z_{0}}\right) Y_{W}(a,z_{1})Y_{W}(b,z_{2})u-z_{0}^{-1}\delta\left(\frac{-z_{2}+z_{1}}{z_{0}}\right)Y_{W}(b,z_{2})Y_{W}(a,z_{1})u\\
		&=z_{2}^{-1}\delta \left(\frac{z_{1}-z_{0}}{z_{2}}\right) Y_{W}(Y(a,z_{0})b,z_{2})u.
		\end{align*}
	\end{enumerate}
	
	A weak $V$-module $W$ is called \name{admissible (or $\N$-gradable)} if $W=\bigoplus_{n\in \N}W(n)$ with $\dim W(n)<\infty$ for each $n\in \N$, and $a_{m}W(n)\subset W(\wt(a)-m-1+n)$ for all homogeneous $a\in V$, $m\in \Z$, and $n\in \N$. For each $n\in \N$, we write $\deg u=n$ for any $u\in W(n)$.
	
	An admissible $V$-module $W=\bigoplus_{n=0}^\infty W(n)$ is called an \name{ordinary $V$-module} if the element $L(0)=\Res_{z} z Y_{W}(\om ,z)$ acts semi-simply on $W$ (that is, $W$ is a semi-simple module under the action of $L(0)$), and there exists a $\la\in \Q$, called the \name{conformal weight} of $W$, such that $W(n)=W_{\la+n}$ is an eigenspace of $L(0)$ of eigenvalue $\la+n\in \Q$ for each $n\in \N$.
	
For each $n\in \N$, we write $\wt (u)=\la+n$ for any $u\in W_{\la+n}$.
\end{df}

Note that $V$ is an ordinary $V$-module with $Y_V$, called the \name{adjoint representation} of $V$.

Let $(W,Y_{W})$ be a weak module over a VOA $V$. Write $Y_{W}(\om ,z)=\sum_{n\in \Z} L(n)z^{-n-2}$. It is proved in \mcite{DLM} that $Y_{W}$ also satisfies the $L(-1)$-derivative property  and the $L(-1)$-bracket derivative property: 
$$ Y_{W}(L(-1)a,z)=\frac{d}{dz}Y_{W}(a,z)=[L(-1),Y_{W}(a,z)], \quad a\in V.$$

Let $(V,Y,\vac,\om)$ be a vertex operator algebra, and $(W,Y_{W})$ be an ordinary $V$-module, with conformal weight $\la\in \Q$. We can construct a semidirect product vertex algebra $V\rtimes W$ (cf. \mcite{L1}, see also the last section in \mcite{FHL}). As a vector space, $V\rtimes W=V\op W$. The vertex operator $Y_{V\rtimes W}$ is given by
\begin{equation}\mlabel{2.11}
Y_{V\rtimes W}(a+u,z)(b+v):=(Y(a,z)b)+(Y_{W}(a,z)v+Y_{WV}^{W}(u,z)b), \quad a, b\in V, u,v\in W.
\end{equation}
Here $Y_{WV}^{W}$ is defined by the skewsymmetry formula
\begin{equation}\mlabel{2.12}
Y_{WV}^{W}(u,z)b=e^{zL(-1) }Y_{W}(b,-z)u, \quad b\in V, v\in W.
\end{equation}
If $W$ only has integral weights, then $(V\rtimes W,Y_{V\rtimes W}, \vac,\om)$ is a vertex operator algebra. In general, $V\rtimes W$ is only a vertex algebra, and it satisfies all the axioms of a VOA except that $L(0)$ only has integral eigenvalues (see Proposition 2.10 in \mcite{L1}).

We also recall the definition of contragredient modules of a VOA $V$. Let $W$ be an admissible $V$-module, and let $W'$ be the graded dual of $W$:
$W'=\bigoplus_{n=0}^{\infty} W(n)^{\ast}.$
Then $(W',Y_{W'})$ is an admissible $V$-module, where
\begin{equation}\mlabel{3.9}
\<Y_{W'}(a,z)f,u\>=\<f,Y_{W}(e^{zL(1)}(-z^{-2})^{L(0)}a,z^{-1})u\>,\quad a\in V, f\in W', u\in W.
\end{equation}
See (5.2.4) in \mcite{FHL}. Moreover, the action of $\mathrm{sl}(2,\C)=\C L(-1)+\C L(0)+\C L(1)$ satisfies the properties
\begin{equation}\mlabel{2.7}
\<L(-1)f,u\>=\<f,L(1)u\>,\quad \<L(0)f,u\>=\<f,L(0)u\>.
\end{equation}

In particular, if $(W,Y_W)$ is an ordinary $V$-module of conformal weight $\la$, then $(W',Y_{W'})$ is also an ordinary $V$-module of the same conformal weight $\la$, and we can construct the semidirect product vertex algebra $V\rtimes W'$.

\subsection{The vertex operator Yang-Baxter equation}
\mlabel{ss:voybe}

In order to properly define the tensor form of the VOYBE, we introduce some new notations of vertex operators based on the definition of contragredient module \meqref{3.9}. Let $(U,Y_U,\vac,\om)$ be a VOA. We define two vertex operators $Y_U'$ and $Y_U'^{\mathrm{op}}$ as follows. For any $a,b\in U$,
\begin{align}
Y_U'(a,z)b:&=Y_U(e^{zL(1)}(-z^{-2})^{L(0)}a,z^{-1})b=\sum_{m\in \Z}(a'_mb) z^{-m-1} ,\mlabel{3.51}\\
Y_U'^{\mathrm{op}}(a,z)b:&= Y_U(e^{-zL(1)}(-z^{-2})^{L(0)}a,-z^{-1})e^{zL(1)}b=\sum_{m\in \Z} (a_m'^{\mathrm{op}}b) z^{-m-1},\mlabel{3.52}
\end{align}

\begin{df}\mlabel{df2.3}
Let $(U,Y_U,\vac,\om)$ be a VOA and $m\in \Z$. Define three \name{$m$-dot products} by \meqref{3.51} and \meqref{3.52} as follows. For $\al, \b\in U$, define
\begin{align*}
\al\cdot _m \b:&=\Res_{z} z^m Y_U(\b,z)\al=\b_m\al,\\
\al\cdot'_m\b:&=\Res_z z^m Y'_U(\al,z)\b=\sum_{j\geq 0}\frac{(-1)^{\wt (\al)}}{j!} (L(1)^{j}\al)_{2\wt(\al)-m-j-2}\b=\al'_m\,\b,\\
\al\cdot'^{\mathrm{op}}_m \b:&=\Res_z z^m Y'^{\mathrm{op}}_U(\b,z)\al=\sum_{i\geq 0}\sum_{j\geq 0} \frac{(-1)^{\wt \b+m+i+1}}{j!i!}(L(1)^{j}\b)_{2\wt(\b)-m-j-i-2}L(1)^i\al=\b'^{\mathrm{op}}_m \al.
\end{align*}
It would be interesting to study algebraic properties of these nonassociative products. 
	
We denote the grading of the VOA $U$ by $U=\bigoplus_{n=0}^\infty U(n)$. By extending the bottom level $U(0)$ by a one-dimensional vector space with a basis $I$, we define the graded space
$$\tilde{U}:=U\op \C I.$$
So $\tilde{U}(0)=U(0) \op \C I$ and $\tilde{U}(n)=U(n)$ for $n\geq 1$. We then extend the above three $m$-dot products to $\tilde{U}$ by taking $I$ to be the identity element:
$$\al\cdot_m I=I\cdot_m \al=\al\cdot'_m I=I\cdot'_m\al=\al\cdot'^{\mathrm{op}}_m I=I\cdot'^{\mathrm{op}}_m\al=\al.$$
\end{df}

\begin{remark}\mlabel{rk3.4}
	For homogeneous elements $\al\in U(s)$ and $\b\in U(t)$, we observe that $	\al\cdot _m \b$, $	\al\cdot'_m\b$ and $\al\cdot'^{\mathrm{op}}_m \b$ are all homogeneous elements in $U$, with
	$\al\cdot_m \b\in U(s+t-m-1)$, $\al\cdot'_m\b\in U(t+m+1-s)$ and $	 \al\cdot'^{\mathrm{op}}_m \b\in U(s+m+1-t)$.
	\end{remark}

We use completion to extend the CYBE for finite-dimensional Lie algebras to VOAs.
\begin{df}\mlabel{df3.5}
	Let $M=\bigoplus_{n =0}^{\infty} M(n)$, $W=\bigoplus_{n=0}^{\infty}W(n)$, and $U=\bigoplus_{n=0}^\infty$ be $\N$-graded vector spaces
	with finite-dimensional graded parts.
	
	\begin{enumerate}
		\item   Define the \name{complete tensor products} $M\mhat{\o} W$ and $M\mhat{\o}W\mhat{\o} U$ by
		\begin{equation}\mlabel{3.11}
			M\mhat{\o} W:=\prod_{p,q=0}^{\infty} M(p)\otimes W(q), \qquad M\mhat{\o}W\mhat{\o} U:=\prod_{p,q,r=0}^\infty M(p)\otimes W(q)\otimes U(r).
		\end{equation}
		\item Let $D(U\mhat{\o}U):=\prod_{t=0}^\infty U(t)\otimes U(t)\subset U\mhat{\o}U$.
		An element $\al$ in $U\mhat{\o}U$ is called {\bf diagonal}. Then $\al=\sum_{t=0}^\infty \sum_{i=1}^{p_t} \al_i^t\otimes \b_i^t\in D(U\mhat{\o}U)$, where $\al_i^t,\b_i^t\in U(t), p_t\geq 1$ for all $t\geq 0$ and $i\geq 1$.
		\item A diagonal element $\al\in D(U\mhat{\o}U)$ is called {\bf skewsymmetric} if $\si(\al)=-\al$, where $$\si:U\mhat{\o}U\ra U\mhat{\o}U, \quad \si\Big(\sum_{t}\sum_{i}\al_i^t\otimes \b_i^t\Big)=\sum_t \sum_{i} \b_i^t\otimes \al_i^t.$$
		A skewsymmetric diagonal element $\al$ can be written as
		$\al=\gamma-\sigma(\gamma),$
		for $\gamma\in D(U\mhat{\o}U).$
		Let $\mathrm{SD}(U\mhat{\o}U)$ denote the subspace of skewsymmetric diagonal elements in $U\mhat{\o} U$.
	\end{enumerate}
\end{df}

Let $(U,Y_U,\vac,\om)$ be a VOA, and $r$ be a diagonal skewsymmetric two-tensor
$$r:=\sum_{t=0}^\infty r^t :=\sum_{t=0}^\infty\sum_{i=1}^{p_t} \al^t_i\otimes \b^t_i-\b^t_i\otimes \al^t_i\in \SD(U\mhat{\o} U).$$
So $r^t:=\sum_{i=1}^{p_t} \al^t_i\otimes \b^t_i-\b^t_i\otimes \al^t_i\in U(t)\o U(t)$ for $t\geq 0$. For any $t,s,q\in \N$, we define the elements $r^t_{12}$, $r^s_{13}$, and $r^q_{23}$ in ${\tilde U}^{\mhat{\o}3}$ as follows. 
\begin{align}
r^t_{12}&:= \sum_{i=1}^{p_t}(\al_i^t\otimes \b_i^t\otimes I-\b_i^t\otimes \al_i^t\otimes I),\mlabel{3.46}\\
r^s_{13}&:=\sum_{k=1}^{p_s}(\al_k^s\otimes I\otimes \b_k^s-\b_k^s\otimes I\otimes \al_k^s),\mlabel{3.47}\\
r^q_{23}&:=\sum_{l=1}^{p_r}(I\otimes \al_l^q\otimes \b_l^q-I\otimes \b_l^q\otimes \al_l^q).\mlabel{3.48}
\end{align}
Then we define $r_{12}:=\sum_{t=0}^\infty r^t_{12}\in \tilde{U}^{\mhat{\o}3}$, $r_{13}:=\sum_{s=0}^\infty r^s_{13}\in \tilde{U}^{\mhat{\o}3}$ and $r_{23}:=\sum_{q=0}^\infty r^q_{23}\in \tilde{U}^{\mhat{\o}3}$.

For $t,s,q\in \N$, and $m\in \Z$, we define the products $r^t_{12}\cdot_m r^s_{13}$, $ r^q_{23}\cdot'_m r^t_{12}$ and $r^s_{13}\cdot'^{\mathrm{op}}_m r^q_{23}$ by multiplying pure tensors in $\tilde{U}^{\mhat{\o}3}$ factor-wise using
the three products in Definition \mref{df2.3}, together with the distributivity for the sums:
	\begin{align*}
	r^t_{12}\cdot_m r^s_{13}&:=\sum_{i,k} ((\al^t_i)\cdot_m \al^s_k\o \b^t_i\o \b^s_k-(\al^t_i)\cdot_m \b^s_k\o \al^t_i\o \b^s_k-(\b^t_i)\cdot_m \al^s_k\o \al^t_i\o \b^s_k \numberthis\mlabel{2.16}\\
	&\quad +(\b^t_i)\cdot_m \b^s_k\o \al^t_i\o \al^s_k),\\
	 r^q_{23}\cdot'_m r^t_{12}&:=\sum_{l,i} (\al^t_i\o (\al^q_l) \cdot'_m \b^t_i\o \b^q_l -\b^t_i\o (\al^q_l)\cdot'_m\al^t_i \o \b^q_l - \al^t_i \o (\b^q_l)\cdot'_m \b^t_i\o \al^q_l \numberthis\mlabel{2.17}\\
	 &\quad +\b^t_i\o (\b^q_l)\cdot'_m\al^t_i\o \al^q_l),\\
	r^s_{13}\cdot'^{\mathrm{op}}_m r^q_{23}&:=\sum_{k,l} (\al^s_k \o \al^q_l \o (\b^s_k)\cdot'^{\mathrm{op}}_m \b^q_l-\al^s_k\o \b^q_l\o (\b^s_k)\cdot'^{\mathrm{op}}_m\al^q_l  -\b^s_k\o \al^q_l\o (\al^s_k)\cdot'^{\mathrm{op}}_m\b^q_l\numberthis \mlabel{2.18}\\
	&\quad + \b^s_k\o \b^q_l\o (\al^s_k)\cdot'^{\mathrm{op}}_m\al^q_l).
	\end{align*}
	Then we define
$$
r_{12}\cdot_m r_{13}:=\sum_{s,t=0}^\infty 	 r^t_{12}\cdot_m r^s_{13}, \quad
r_{23}\cdot'_m r_{12}:=	 \sum_{q,t=0}^\infty r^q_{23}\cdot'_m r^t_{12},\ \ \mathrm{and}\ \  r_{13}\cdot'^{\mathrm{op}}_m r_{23}:=\sum_{s,q=0}^\infty 	 r^s_{13}\cdot'^{\mathrm{op}}_m r^q_{23}.
$$
\begin{lm}\mlabel{lm3.5}
$	r_{12}\cdot_m r_{13}$, $r_{23}\cdot'_m r_{12}$ and $r_{13}\cdot'^{\mathrm{op}}_m r_{23}$ are well-defined elements in $U^{\mhat{\o}3}$. Let $\al=r_{12}\cdot_m r_{13}-r_{23}\cdot'_m r_{12}+r_{13}\cdot'^{\mathrm{op}}_m r_{23}\in U^{\mhat{\o}3}$. Then we have
\begin{equation}\mlabel{3.14'}
\al=\sum_{s,t\ge 0,\ s+t\geq m+1}\al_{s,t}=\sum_{s,t\ge 0,\ s+t\geq m+1} (r^s_{12}\cdot_m r_{13}^t-r_{23}^t\cdot'_m r_{12}^{t+s-m-1}+r_{13}^{t+s-m-1}\cdot'^{\mathrm{op}}_{m} r_{23}^s),
\end{equation}
where $\al_{s,t}:=(r^s_{12}\cdot_m r_{13}^t-r_{23}^t\cdot'_m r_{12}^{t+s-m-1}+r_{13}^{t+s-m-1}\cdot'^{\mathrm{op}}_{m} r_{23}^s)\in U(t+s-m-1)\o U(s)\o U(t)$ for each pair $(s,t)\in \N\times \N$ such that $s+t\geq m+1$.
	\end{lm}
\begin{proof}
	By Remark \mref{rk3.4}, we have
	\begin{align}
r^t_{12}\cdot_m r^s_{13}& \in U(s+t-m-1)\o U(t)\o U(s),\mlabel{3.15}\\
 r^q_{23}\cdot'_m r^t_{12}&\in U(t)\o U(t+m+1-q)\o U(q), \mlabel{3.16}\\
 	r^s_{13}\cdot'^{\mathrm{op}}_m r^q_{23}& \in U(s)\o U(q)\o U(s+m+1-q).\mlabel{3.17}
\end{align}
By Definition \mref{df3.5}, we have $r_{12}\cdot_m r_{13}=\sum_{t,s=0}^\infty r^t_{12}\cdot_m r^s_{13}\in \prod_{t,s=0}^\infty U(s+t-m-1)\o U(t)\o U(s)$, which is a linear subspace of $U^{\mhat{\o}3}$. Thus, $r^t_{12}\cdot_m r^s_{13}$, and similarly $ r^q_{23}\cdot'_m r^t_{12}$ and $	 r^s_{13}\cdot'^{\mathrm{op}}_m r^q_{23}$, are well-defined elements in $U^{\mhat{\o}3}$.

Note that $U(s+t-m-1)=0$ if $s+t-m-1<0$. Exchanging $s$ and $t$ in \meqref{3.15},  we have $r_{12}\cdot_m r_{13}=\sum_{s,t=0,\ s+t\geq m+1}^\infty r^s_{12}\cdot_m r_{13}^t$, and $r^s_{12}\cdot_m r_{13}^t\in U(t+s-m-1)\o U(s)\o U(t)$ for each pair $(s,t)$.
Moreover, we observe that there is a bijection
\begin{align*}
\{(t,q)\in \N\times \N: t+m+1-q\geq 0\}&\ra \{ (s_0,t_0)\in \N\times \N: s_0+t_0-m-1\geq 0\},\\
(t,q)&\mapsto (s_0,t_0)=(t+m+1-q,q),
\end{align*}
whose inverse is given by $ (s_0,t_0)\mapsto (t,q):=(s_0+t_0-m-1,t_0)$. Thus we can use the changes of variable $t$ and $q$ in \meqref{3.16} by $s_0+t_0-m-1$ and $t_0$ respectively, and obtain
\begin{equation}\mlabel{3.18}
r_{23}\cdot'_m r_{12}=\sum_{q,t\geq 0,\ t+m+1-q\geq 0}  r^q_{23}\cdot'_m r^t_{12}=\sum_{s_0,t_0\geq 0,\ s_0+t_0-m-1\geq 0} r_{23}^{t_0}\cdot'_m r_{12}^{t_0+s_0-m-1},
\end{equation}
and $r_{23}^{t_0}\cdot'_m r_{12}^{t_0+s_0-m-1}\in U(t_0+s_0-m-1)\o U(s_0)\o U(t_0)$ for each pair $(s_0,t_0)$. Finally, there is a bijection
\begin{align*}
\{(q,s)\in \N\times \N: s+m+1-q\geq 0\}&\ra \{(s_1,t_1): s_1+t_1-m-1\geq 0 \},\\
(q,s)&\mapsto (s_1,t_1)=(q,s+m+1-q),
\end{align*}
whose inverse is given by $(s_1,t_1)\mapsto (q,s)=(s_1,t_1+s_1-m-1)$. Then changing the variables $s$ and $q$ in \meqref{3.17} by $s_1$ and $t_1+s_1-m-1$ respectively, we obtain
\begin{equation}\mlabel{3.19}
	r_{13}\cdot'^{\mathrm{op}}_m r_{23}=\sum_{q,s\geq 0,\ s+m+1-q\geq 0} 	 r^s_{13}\cdot'^{\mathrm{op}}_m r^q_{23}=\sum_{s_1,t_1\geq 0,\ s_1+t_1-m-1\geq 0}	 r^{t_1+s_1-m-1}_{13}\cdot'^{\mathrm{op}}_m r^{s_1}_{23},
\end{equation}
and $r^{t_1+s_1-m-1}_{13}\cdot'^{\mathrm{op}}_m r^{s_1}_{23}\in U(t_1+s_1-m-1)\o U(s_1)\o U(t_1)$. Now \meqref{3.14'} follows after we replace the variables $(s_0,t_0)$ in \meqref{3.18} and $(s_1,t_1)$ in \meqref{3.19} by $(s,t)$.
\end{proof}

\begin{remark}
In fact, we can also define $r_{12}$, $r_{13},$ and $r_{23}$ by \meqref{3.46}-\meqref{3.48} for arbitrary diagonal elements $r\in D(U\mhat{\o}U)$ (not necessarily skewsymmetric ones), and by a proof similar to the one for Lemma \mref{lm3.5}, the products $r_{12}\cdot_m r_{13}$, $r_{23}\cdot'_m r_{12}$ and $r_{13}\cdot'^{\mathrm{op}}_m r_{23}$ are well defined. But we will be focusing on the skewsymmetric two tensors $r$ for the rest of the paper.
\end{remark}

Now we give the definition of the (classical) Yang-Baxter equation for VOAs.
\begin{df}\mlabel{df:VOYBE}
Let $(U,Y_U,\vac,\om)$ be a VOA and $r\in \SD(U\mhat{\o} U)$ be skewsymmetric.
	\begin{enumerate}
\item Let $m\in \Z$ be a fixed integer.	$r$ is called a \name{skewsymmetric solution} to the {\bf $m$-vertex operator Yang-Baxter equation ($m$-VOYBE)} if
	\begin{equation}\mlabel{3.22}
	r_{12}\cdot_m r_{13}- r_{23}\cdot'_m r_{12}+r_{13}\cdot'^{\mathrm{op}}_m r_{23}=0.
	\end{equation}
\item $r$ is called a skewsymmetric solution to the {\bf vertex operator Yang-Baxter equation (VOYBE)} if it is a solution to $m$-VOYBE for every $m\in \N$. In other words, we have
\begin{equation}\mlabel{3.23}
	r_{12}\cdot_z r_{13}-r_{23}\cdot_z r_{12}+r_{13}\cdot_z r_{23}=0,
\vspace{-.2cm}
\end{equation}
where we denote \vspace{-.1cm}
$$r_{12}\cdot_z r_{13}:= \sum_{m\in \Z} (r_{12}\cdot_m r_{13}) z^{-m-1},  r_{23}\cdot_z r_{12}:=\sum_{m\in \Z} (r_{23}\cdot'_m r_{12}) z^{-m-1}, r_{13}\cdot_z r_{23}:=\sum_{m\in \Z}(r_{13}\cdot'^{\mathrm{op}}_m r_{23})z^{-m-1}.$$
	\end{enumerate}
\end{df}
See below for their operator forms and see Section~\mref{s:cybe} for the relation with the CYBE.

\vspace{-.2cm}
\subsection{Relative Rota-Baxter operators}
\mlabel{ss:relrbo}

Now we introduce the notion of relative Rota-Baxter operators for VOAs. It is a generalization of the $R$-matrix of VOAs introduced by Xu~\mcite{X1} and serves here as the operator form of the classical Yang-Baxter equation for VOAs, in analog to the case of the CYBE for Lie algebras~\mcite{K1,S}.

Let $(V,Y,\vac,\om)$ be a VOA, $(W,Y_W)$ be an admissible $V$-module. Let $m\in \Z$, $a\in V$ and $u\in W$. Recall
$a_m u=\Res_{z}z^m Y_W(a,z)u$, and define an operator
\begin{equation}\mlabel{3.7}
u(m): V\to W, \quad  u(m)a:=\Res_{z} z^m Y_{WV}^W(u,z)a.
\end{equation}

\begin{df}\mlabel{df3.1}
Let $(V,Y,\vac,\om)$ be a vertex operator algebra and $(W,Y_{W})$ be an admissible $V$-module. Let
$T:W\ra V$ be a linear map.
\begin{enumerate}
	\item For $m\in \Z$, $T$ is called an \name{$m$-relative Rota-Baxter operator ($m$-relative RBO)}, or an \name{$m$-$\mathcal{O}$-operator}, of the VOA $V$ associated to the $V$-module $W$, if
	\begin{equation}\mlabel{3.8}
		T(u)_m T(v) =T(T(u)_mv)+ T(u(m)T(v)), \quad u, v\in W.
	\end{equation}
	\item $T$ is called a \name{relative Rota-Baxter operator (relative RBO)}, or an \name{$\mathcal{O}$-operator}, of the VOA $V$ associated to the $V$-module $W$, if it is an $m$-relative RBO for every $m\in \Z$, that is, the following equation holds.
	\begin{equation}\mlabel{3.9b}
		Y(Tu,z) Tv=T(Y_{W}(Tu,z)v+Y_{WV}^{W}(u,z)Tv),\quad u, v\in W.
	\end{equation}
Here $Y_{WV}^W$ is given by the skewsymmetry formula \meqref{2.12}.
	
	\item An $m$-relative RBO $T: W\ra V$ is called \name{homogeneous of degree $N\in \Z$} if $T(W(n))\ssq V_{n+N}$ for each $n\in \N$. A degree $0$ relative $m$-RBO is called \name{level preserving}.
\item When $W$ is the adjoint representation of $V$ on itself, then an ($m$-)relative RBO is simply called an \name{($m$-)RBO} on $V$.
\end{enumerate}
\end{df}

For a vector space $A$ with a binary operation $\cdot$, a \name{Rota-Baxter oeprator (RBO)} (of weight $0$) on $A$ is a linear operator $P:A\to A$ such that
$$ P(u)\cdot P(v)=P(u\cdot P(v))+P(P(u)\cdot v), \quad u, v\in A.$$
This applies in particular when $\cdot$ is the associative product or the Lie bracket where the notion was introduced by G.~Baxter~\mcite{BaG} and Semenov-Tian-Shanski~\mcite{S} respectively.

Motivated by the operator form of the CYBE whose solutions are called classical $r$-matrices, the notion of a classical $R$-matrix for a vertex operator algebra $V$ was introduced in~\mcite{X1}. This notion is precisely the one for RBOs on $V$. As we will see in Corollary~\mref{co:rrb}, RBOs on $V$ are indeed solutions of the VOYBE, justifying the term of $R$-matrix.

For a given $\lambda\in \C$, there is a more general notion of $m$-relative RBO of weight $\lambda$, of which the above notion is the special case when $\lambda=0$.  In this paper, an ($m$-)relative RBO is always assumed to have weight $0$.
The general case is considered in~\mcite{BGLW}.

\begin{example}\mlabel{ex2.10}
Let $V=V_0\op V_+$ be a CFT-type VOA, and let $P:V_1\ra V_1$ be a RBO of the Lie algebra $V_1$ of weight $0$. Extend $P$ to $T: V\ra V$ by
\begin{equation}\mlabel{2.11'}
	T(\vac):=\mu \vac,\quad T|_{V_1}:=P,\quad \mathrm{and}\quad  T|_{V_n}:=0,\  n\geq 2,
\end{equation}
where $\mu\in \C$ is a fixed number. We claim that $T:V\ra V$ is a level-preserving $0$-relative RBO. Indeed, clearly $T$ is level-preserving. For $a\in V$, we have
\begin{align*}
	T(\vac)_0 T(a)&=0=T(T(\vac)_0 a)+T(\vac_0 T(a)),\\
	T(a)_0 T(\vac)&= 0=T(a_0 T(\vac))+T(T(a)_0\vac),
\end{align*}
since $\vac_0 a=a_0 \vac=0$. On the other hand, for homogeneous elements $a,b\in V_+$, if either $\wt (a)>1$  or $\wt (b)>1$, then by \meqref{2.11'} and the fact that $T(a)_0 b$ and $a_0T(b)$ are contained in $V_{\wt (a)+\wt (b)-1}$, we have
$T(a)_0 T(b)=0=T(T(a)_0b)+T(a_0 T(b)).$
Finally, if $a,b\in V_1$, then clearly we have $T(a)_0 T(b)=T(T(a)_0b)+T(a_0 T(b))$ since $T=P$ on $V_1$. Thus $T:V\ra V$ is a level-preserving $0$-relative RBO.
\end{example}
\begin{example}
	Let $V=M_{\hat{\h}}(1,0)$ be the rank-one Heisenberg VOA (cf. \mcite{FLM}), where $\h=\C \al$ and $(\al|\al)=1$. Recall that $V_1=\C \al(-1)\vac$ and $V_2=\C \al(-1)\al\op \C \al(-2)\vac$.
	Define $T:V\ra V$ by
	\begin{align*}
	&T(\vac)=T(\al(-1)\vac)=0,\quad  T|_{V_n}=0,\quad n\geq 3,\\
	&T(\al(-1)\al)=\al(-1)\al+\al(-2)\vac,\quad \mathrm{and}\quad T(\al(-2)\vac)=-\al(-1)\al-\al(-2)\vac.
	\end{align*}
	For $a,b\in V_2$, since $T(a)_1b$ and $a_1T(b)$ are contained in $V_{2}$, by a similar argument as for the previous example, we can easily show that $T:V\ra V$ is a level-preserving $1$-RBO.
\end{example}
\begin{example}
	Let $V=M_{\hat{\h}}(1,0)$ be the rank-one Heisenberg VOA and let $W=M_{\hat{\h}}(1,\la)$. Recall \mcite{FLM} that $W=M_{\hat{\h}}(1,0)\o \C e^\la$ with $W(0)=\C e^\la$ and $W(1)=\C \al(-1)e^\la$. Define $T:W\ra V$ by
	\begin{equation}\mlabel{2.12'}
	T(e^\la)=0,\quad T(\al(-1)e^\la)=\vac,\quad \mathrm{and}\quad T|_{W(n)}=0,\quad n\geq 2.
	\end{equation}
	Then $T$ is homogeneous of degree $-1$. If $u,v\in W$ are homogeneous, then $(Tu)_{0}v$ and $u(0)T(v)$ are contained in $W(\deg u-2+\deg v)$, and by \meqref{2.12'}, we have
	$(Tu)_{0} (Tv)=0=T((Tu)_{0}v+u(0)Tv).$
	Thus, $T:W\ra V$ is a degree $-1$ homogeneous $0$-relative RBO.
\end{example}

\subsection{From solutions of the VOYBE to relative RBOs}

Our definition of the classical Yang-Baxter equation for VOAs is motivated by the axiom of a relative RBO for VOAs. Indeed, let $(U,Y_U, \vac, \om)$ be a VOA, and let $U'=\bigoplus_{t=0}^\infty U(t)^\ast$ be the contragredient module of $U$~\mcite{FHL}. For the notations in Definition \mref{df3.5}, we have the following identifications of vector spaces:
\begin{equation}\mlabel{3.24'}
\Phi: D(U\mhat{\o} U)=\prod_{t=0}^\infty U(t)\o U(t)\cong \prod_{t=0}^\infty \Hom(U(t)^\ast,U(t))\cong \Hom_{\mathrm{LP}}(U',U),
\end{equation}
where $\Hom_{\mathrm{LP}}(U',U)\subset \Hom(U',U)$ denotes the subspace of level-preserving linear maps.

Let $r=\sum_{t=0}^\infty r^t=\sum_{t=0}^\infty\sum_{i=1}^{p_t} \al^t_i\otimes \b^t_i-\b^t_i\otimes \al^t_i\in\SD(U\mhat{\o} U)$ be a diagonal skewsymmetric element, and denote $\Phi(r)\in \Hom_{\mathrm{LP}}(U',U)$ by $T_r$. Then $T_r$ is given by
	\begin{equation}\mlabel{3.23b}
T_r(f):=\sum_{i=1}^{p_t} \al^t_i \<f,\b^t_i\>-\b^t_i\<f,\al^t_i\>,\quad t\in \N, f\in U(t)^\ast.
\end{equation}

\begin{thm}\mlabel{thm:main1}
With the setting as above, for a given $m\in \Z$, $r$ is a skewsymmetric solution to the $m$-VOYBE if and only if $T_r: U'\ra U$ is a level-preserving $m$-relative RBO, that is, for any $f,g\in U'$, the following equation holds:
 \begin{equation}\mlabel{3.24}
 T_r(f)_m T_r(g)=T_r(T_r(f)_m g)+T_r(f(m)T_r(g)).
 \end{equation}
In particular, $r$ is a skewsymmetric solution to the VOYBE if and only if $T_r:U'\ra U$ is a level-preserving relative RBO.
	\end{thm}
\begin{proof}
	For given $a,b,c\in U$, with $b,c$ homogeneous, define a linear functional $$a\otimes b\otimes c: U'\otimes U'\ra U, \quad (a\otimes b\otimes c)(g\otimes f):=a\<g,b\>\<f,c\>,$$
	for homogeneous $f,g\in U'$.
	
	For $f\in U(t)^\ast$ and $g\in U(s)^\ast$, by \meqref{3.23b}, we have $T_r(g):=\sum_{j=1}^{p_s} \al^s_j \<g,\b^s_j\>-\b^s_j\<g,\al^s_j\>$. For a given $m\in \Z$, by Definition \mref{df2.3} and \meqref{3.15}, we derive
	\begin{align*}
T_r(f)_mT_r(g)&=\sum_{i,j}(\al^t_i \<f,\b^t_i\>-\b^t_i\<f,\al^t_i\>)_m(\al^s_j \<g,\b^s_j\>-\b^s_j\<g,\al^s_j\>)\\
&= \sum_{i,j} (\al^t_i)_m \al^s_j \<f,\b^t_i\>\<g,\b^s_j\>-\sum_{i,j}(\al^t_i)_m \b^s_j\<f,\b^t_i\>\<g,\al^s_j\>\\
&\ -\sum_{i,j} (\b^t_i)_m \al^s_j \<f,\al^t_i\>\<g,\b^s_j\>+\sum_{i,j} (\b^t_i)_m \b^s_j \<f,\al^t_i\>\<g,\al^s_j\>\\
&=\sum_{i,j} ((\al^s_j)\cdot_m \al^t_i\o \b^s_j\o \b^t_i)(g\o f)-\sum_{i,j} ((\b^s_j)\cdot_m \al^t_i\o \al^s_j\o \b^s_i)(g\o f)\\
&\ -\sum_{i,j} ((\al^s_j)\cdot_m \b^t_i \o \b^s_j\o \al^t_i)(g\o f)+\sum_{i,j} ((\b^s_j)\cdot_m \b^t_i\o \al^s_j\o \al^t_i)(g\o f)\\
	&=(r^s_{12}\cdot_m r^t_{13})(g\otimes f).
\end{align*}
By Definition \mref{df2.3} and \meqref{3.9b}, we have $\<a_m f, b\>=\<f,a'_mb\>=\<f,a\cdot'_mb\>$ for $a,b\in U$ and $f\in U'$. Since $T_r(T_r(f)_m g)\in U(t+s-m-1)^\ast$, by Remark \mref{rk3.4}, \meqref{3.23b}, and \meqref{3.16}, we obtain
	\begin{align*}	
T_r(T_r(f)_mg)&=\sum_k \al_k^{t+s-m-1}\<T_r(f)_mg, \b^{t+s-m-1}_k\>- \sum_k \b^{t+s-m-1}_k\<T_r(f)_mg,\al^{t+s-m-1}_k\>\\
&=\sum_{k,i} \al_k^{t+s-m-1}\<g, (\al^t_i)\cdot'_m \b_k^{t+s-m-1}\>\<f,\b^t_i\>-\sum_{k,i} \al_k^{t+s-m-1}\<g,(\b^t_i)\cdot'_m \b_k^{t+s-m-1}\>\<f,\al^t_i\>\\
&\ -\sum_{k,i} \b_k^{t+s-m-1} \<g, (\al^t_i)\cdot'_m \al_k^{t+s-m-1}\>\<f,\b^t_i\>+\sum_{k,i}\b_k^{t+s-m-1}\<g,(\b^t_i)\cdot'_m \al_k^{t+s-m-1}\>\<f,\al^t_i\>\\
&=\sum_{k,i} ((\al_k^{t+s-m-1}\o (\al^t_i)\cdot'_m \b_k^{t+s-m-1}\o \b^t_i)- (\al^{t+s-m-1}_k\o (\b^t_i)\cdot'_m \b_k^{t+s-m-1}\o \al^t_i))(g\o f)\\
&\ -\sum_{k,i} ((\b^{t+s-m-1}_k\o (\al^t_i)\cdot'_m \al_k^{t+s-m-1}\o \b^t_i)+(\b^{t+s-m-1}_k\o (\b^t_i)\cdot'_m \al_k^{t+s-m-1}\o \al^t_i))(g\o f)\\
	&=(r^t_{23}\cdot'_{m} r^{t+s-m-1}_{12})(g\otimes f).
	\end{align*}
	Finally, by Definition \mref{df2.3}, \meqref{2.12}, \meqref{3.9b}, and the fact that $\<L(-1)f,a\>=\<f,L(-1)a\>$ for $f\in U'$ and $a\in U$ (see~\cite[5.2.10]{FHL}), we have
	\begin{align*}
	\<f(m)a,b\>&=\Res_{z} z^m \<Y_{U'U}^{U'}(f,z)a,b\>=\Res_{z} z^m \<e^{zL(-1)}Y_{U'}(a,-z)f,b\>\\
	&=\Res_{z} z^m \<f, Y_U(e^{-zL(1)}(-z^{-2})^{L(0)}a,-z^{-1})e^{zL(1)}b\>=\Res_{z}z^m \<f, Y_{U}'^{\mathrm{op}}(a,z)b\>\\
	&=\<f,b\cdot'^{\mathrm{op}}_m a\>,\quad f\in U', a, b\in U.
	\end{align*}
Since $Y_{U'U}^{U'}$ is an intertwining operator (see Section 5.4 in \mcite{FHL}) and $T_r$ is level-preserving, we have $T_r(f(m) T_r(g))\in U(t+s-m-1)^\ast$. Then
\begin{align*}	
&T_r(f(m) T_r(g))=\sum_{k}\al^{t+s-m-1}_k \<f(m) T_r(g),\b^{t+s-m-1}_k\>-\b^{t+s-m-1}_k\<f(m)T_r(g),\al^{t+s-m-1}_k\>\\
=&\sum_{k,j} \al^{t+s-m-1}_k \<f, (\b^{t+s-m-1}_k)\cdot'^{\mathrm{op}}_m \al^s_j\>\<g,\b^s_j\>-\sum_{k,j}\al^{t+s-m-1}_k\<(\b^{t+s-m-1}_k)\cdot'^{\mathrm{op}}_m \b^s_j\>\<g,\al^s_j\>\\
&\ -\sum_{k,j} \b^{t+s-m-1}_k\<f,(\al^{t+s-m-1}_k)\cdot'^{\mathrm{op}}_m \al^s_j\>\<g,\b^s_j\>+\sum_{k,j}\b^{t+s-m-1}_k\<f, (\al^{t+s-m-1}_k)\cdot'^{\mathrm{op}}_m \b^s_j\>\<g,\al^s_j\>\\
=&\sum_{k,j}((\al^{t+s-m-1}_k\o \b^s_j \o (\b^{t+s-m-1}_k)\cdot'^{\mathrm{op}}_m\al^s_j)-(\al^{t+s-m-1}_k\o \al^s_j\o (\b^{t+s-m-1}_k)\cdot'^{\mathrm{op}}_m\b^s_j)(g\o f)\\
&\ -\sum_{k,j} ((\b^{t+s-m-1}_k\o \b^s_j \o (\al^{t+s-m-1}_k)\cdot'^{\mathrm{op}}_m\al^s_j)+(\b^{t+s-m-1}_k\o \al^s_j\o (\al^{t+s-m-1}_k)\cdot'^{\mathrm{op}}_m\b^s_j))(g\o f)\\
=&-(r^{t+s-m-1}_{13}\cdot'^{\mathrm{op}}_{m}r^s_{23})(g\otimes f).
		\end{align*}
	Thus, for $f\in U(t)^\ast$ and $g\in U(s)^\ast$, we have
\begin{equation}\mlabel{3.25}
T_r(f)_m T_r(g)-T_r(T_r(f)_m g)-T_r(f(m)T_r(g))
= (r^s_{12}\cdot_m r_{13}^t-r_{23}^t\cdot'_m r_{12}^{t+s-m-1}+r_{13}^{t+s-m-1}\cdot'^{\mathrm{op}}_{m} r_{23}^s)(g\o f),
\end{equation}
which is contained in $ U(t+s-m-1)$. Now for $T_r$ to satisfy \meqref{3.24} means that $T_r$ satisfies
$$T_r(f)_m T_r(g)-T_r(T_r(f)_m g)-T_r(f(m)T_r(g))=0, \quad  f\in U(t)^\ast, g\in U(s)^\ast, s,t\geq 0.$$
By \meqref{3.25}, this holds if and only if
\begin{align*}
 &\<T_r(f)_m T_r(g)-T_r(T_r(f)_m g)-T_r(f(m)T_r(g)),h\>\\
&= \<(r^s_{12}\cdot_m r_{13}^t-r_{23}^t\cdot'_m r_{12}^{t+s-m-1}+r_{13}^{t+s-m-1}\cdot'^{\mathrm{op}}_{m} r_{23}^s)(g\o f), h\>\\
&=\<(r^s_{12}\cdot_m r_{13}^t-r_{23}^t\cdot'_m r_{12}^{t+s-m-1}+r_{13}^{t+s-m-1}\cdot'^{\mathrm{op}}_{m} r_{23}^s) ,g\o f\o h\>=0,
\end{align*}
for $f\in U(t)^\ast$, $g\in U(s)^\ast$,  $h\in U(t+s-m-1)^\ast$ and  $s,t\geq 0$. Since the element  $r^s_{12}\cdot_m r_{13}^t-r_{23}^t\cdot'_m r_{12}^{t+s-m-1}+r_{13}^{t+s-m-1}\cdot'^{\mathrm{op}}_{m} r_{23}^s \in U^{\mhat{\o}3}$ is homogeneous in $U(s+t-m-1)\o U(s)\o U(t)$ by Lemma \mref{lm3.5}, and the space $U(s+t-m-1)\o U(s)\o U(t)$ is finite dimensional, it follows that $T_r$ satisfies \meqref{3.24} if and only if $r^s_{12}\cdot_m r_{13}^t-r_{23}^t\cdot'_m r_{12}^{t+s-m-1}+r_{13}^{t+s-m-1}\cdot'^{\mathrm{op}}_{m} r_{23}^s=0$ for $(s,t)\in \N\times \N$ with $s+t\geq m+1$. The latter holds if and only if
$	r_{12}\cdot_m r_{13}- r_{23}\cdot'_m r_{12}+r_{13}\cdot'^{\mathrm{op}}_m r_{23}=0$
by Lemma \mref{lm3.5}.	
	\end{proof}

\delete{
We have the following direct consequence from the proof of Theorem \mref{thm:main1}.
\begin{coro}\mlabel{coro2.15}
	With the setting as above, for fixed integers $m\in \Z$ and $s,t\in \N$, the operator $T_{r}$ satisfies \meqref{3.24} for $f\in U(s)^\ast$ and $g\in U(t)^\ast$ if and only if the homogeneous components of $r=\sum_{t=0}^\infty r^t$  satisfy
	\begin{equation}
	 r^s_{12}\cdot_m r_{13}^t-r_{23}^t\cdot'_m r_{12}^{t+s-m-1}+r_{13}^{t+s-m-1}\cdot'^{\mathrm{op}}_{m} r_{23}^s=0
	\end{equation}
	in the homogeneous subspace $U(s+t-m-1)\o U(s)\o U(t)$ of $U^{\mhat{\o}3}$.
	\end{coro}
}

When $U'\cong U$ as a $U$-module, fix a $U$-module isomorphism $\varphi: U\ra U'$. Then by  in \cite[Prop.~5.3.6]{FHL}, there exists a non-degenerate symmetric invariant bilinear form $(\cdot\,|\,\cdot):U\times U\ra \C$ such that
\begin{equation}\mlabel{2.33}
\<\varphi(a),b\>=(a|b),\quad a,b\in U.
\end{equation}
Moreover, since $\varphi(Y_U(a,z)b)=Y_{U'}(a,z)\varphi(b)$ and $\varphi L(-1)=L(-1)\varphi$, if $T:U'\ra U$ is a relative RBO, then for $a,b\in U$, we have
\begin{align*}
Y_U((T\varphi)(a),z)(T\varphi)(b)&=T(Y_{U'}((T\varphi)(a),z)\varphi(b))+T(Y_{U'U}^{U'}(\varphi(a),z)(T\varphi)(b))\\
&=T\varphi(Y_U((T\varphi)(a),z)b)+T\left(e^{zL(-1)}Y_{U'}((T\varphi)(b),-z)\varphi(a)\right)\\
&=(T\varphi) (Y_U((T\varphi)(a),z)b)+ (T\varphi)(Y_U(a,z)(T\varphi)(b)).
\end{align*}
Thus, $T\varphi: U\ra U$ is just an RBO on $U$. Clearly, the converse is also true, and by taking $\Res_{z} z^m$, we find that $T: U'\ra U$ is an $m$-relative RBO if and only if $T\varphi: U\ra U$ is an $m$-RBO.
Moreover, for $T_r:U'\ra U$ given by \meqref{3.23b}, denote $T_r\varphi$ by $\tilde{T_r}$. Then by \meqref{2.33} we have
\begin{equation}\mlabel{2.34}
\widetilde{T_r}(a)=\sum_{i=1}^{p_t}\al^t_i (a\,|\,\b^t_i)-\b^t_i(a\,|\,\al^t_i),\quad a\in U(t), t\in \N.
\end{equation}
Hence we have the following conclusion.
\begin{coro} \mlabel{co:rrb}
	Let $\varphi: U\ra U'$ be an isomorphism as modules over the VOA $U$ and $m\in \Z$. Let $r=\sum_{t=0}^\infty\sum_{i=1}^{p_t} \al^t_i\otimes \b^t_i-\b^t_i\otimes \al^t_i\in \SD(U\mhat{\o} U)$ be a diagonal skewsymmetric element. Then $r$ is a solution to the $m$-VOYBE if and only if $\widetilde{T_r}: U\ra U$ defined by \meqref{2.34} is an $m$-RBO on the VOA $U$.
	\end{coro}

In fact, Theorem \mref{thm:main1} gives rise to a one-to-one correspondence between the solutions to the VOYBE and certain relative RBOs. We end the section by making this precise.
\begin{df}\mlabel{df2.16}
Let $(U,Y_U,\vac,\om)$ be a VOA. We call a level-preserving linear map $T:U'\ra U$ \name{skewsymmetric} if $T$ satisfies
$$\<T(f),g\>=-\<f,T(g)\>, \quad t\in \N, f,g\in U(t)^\ast.$$
Let $ \Hom^{\mathrm{sk}}_{\mathrm{LP}}(U',U)$ denote the subspace of skewsymmetric  level-preserving linear maps.
\end{df}

It is easy to check that $T_r:U'\ra U$ from \meqref{3.23b} is skewsymmetric. We also state the simple fact:
\begin{lm}\mlabel{2.41}
	The linear bijection \meqref{3.24'} restricts to a linear bijection
	\begin{equation}\mlabel{2.40}
\Phi: \SD(U\mhat{\o} U)\cong \Hom^{\mathrm{sk}}_{\mathrm{LP}}(U',U),\quad r\mapsto T_r,
	\end{equation}
	where $T_r$ is defined by \meqref{3.23b}. The inverse of $\Phi$ is given by
	\begin{equation}\mlabel{eq:43}
	\Psi: \Hom^{\mathrm{sk}}_{\mathrm{LP}}(U',U)\ra \SD(U\mhat{\o} U),\quad T\mapsto \frac{1}{2}\sum_{t=0}^\infty\sum_{i=1}^{p_t}\left( T((v_i^t)^\ast)\otimes v^t_i-v^t_i\otimes T((v^t_i)^\ast)\right),
	\end{equation}
	where $\{v^t_1,\dots, v^t_{p_t}\}$ is a basis for $U(t)$, and $\{(v^t_1)^\ast,\dots, (v^t_{p_t})^\ast\}$ is the dual basis of $U(t)^\ast$ for $t\in \N$.
	\end{lm}
Then by Theorem \mref{thm:main1}, we obtain
\begin{coro}\mlabel{co:2.42}
For $m\in \Z$, let $\SD_{\mathrm{sol}}(U\mhat{\o} U)(m)$ denote the set of skewsymmetric solutions to the $m$-VOYBE, and let $ \mathrm{RBO}_{\mathrm{LP}}^{\mathrm{sk}}(U',U)(m)$ denote the set of level-preserving skewsymmetric $m$-relative RBOs.
Then $\Phi$ in \meqref{2.40} restricts to a bijection \begin{equation}\mlabel{2.42}
\Phi: \SD_{\mathrm{sol}}(U\mhat{\o} U)(m)\leftrightarrow \mathrm{RBO}_{\mathrm{LP}}^{\mathrm{sk}}(U',U)(m),\quad r\mapsto T_r.
\end{equation}
\end{coro}

\section{Solving the vertex operator Yang-Baxter equation by relative RBOs}
\mlabel{s:ofvoybe} Let $(U,Y_U,\vac,\om)$ be a VOA. Theorem \mref{thm:main1} indicates that we can construct an $m$-relative RBO $T_r: U'\ra U$ from a skewsymmetric solution of the $m$-VOYBE in the VOA $U$. In the other direction, for a skewsymmetric level-preserving $m$-relative RBO $T: U'\ra U$, by the correspondence \meqref{2.42}, we can also construct a solution $\Psi(T)$ of the $m$-VOYBE in the VOA $U$, where $\Psi$ is given by \meqref{eq:43}.
Now we show that, from certain not necessarily skewsymmetric $m$-relative RBOs for a VOA associated to arbitrary admissible modules, we can still construct solutions to the $m$-VOYBE in a larger VOA, similar to the case of Lie algebras~\mcite{B2}.

\subsection{Actions on the contragredient modules}
\mlabel{ss:3.1}
Let $(V,Y,\vac,\om)$ be a VOA and $(W,Y_W)$ be an ordinary $V$-module of conformal weight $\la\in \Q$. Under the assumption of $\la=0$, we solve the VOYBE from some special relative RBOs $T:W\ra V$.

With the assumption, the contragredient module $(W',Y_{W'})$ is an ordinary $V$-module of the same conformal weight $0$, and the semidirect product $(V\rtimes W',Y_{V\rtimes W'},\vac)$ as recalled in Section \mref{ss:voybe} is a vertex operator algebra, where $Y_{V\rtimes W'}$ is given by \meqref{2.11}. We can write it as
\begin{equation}\mlabel{3.28}
Y_{V\rtimes W'}(a+f,z)(b+g)=(Y_V(a,z)b)+(Y_{W'}(a,z)g+Y_{W'V}^{W'}(f,z)b),\quad a,b\in V, f,g\in W',
\end{equation}
where $Y_{W'}$ is defined by the adjoint formula \meqref{3.9b}, and $Y_{W'V}^{W'}$ is defined by the skewsymmetry formula \meqref{2.12} with respect to $Y_{W'}$:
\begin{equation}\mlabel{3.44'}
Y_{W'V}^{W'}(f,z)a=e^{zL(-1)} Y_{W'}(a,-z)f.
\end{equation}

 We denote the VOA $(V\rtimes W',Y_{V\rtimes W'},\vac,\om)$ by $(U,Y_U,\vac,\om)$. Since the conformal weight of $W'$ is $0$, the admissible gradation $W'=\bigoplus_{n=0}^\infty W(n)^\ast$ is the same as the $L(0)$-eigenspace gradation, and the gradation on $U$ is given by
 $$U(n):=V_n\op W(n)^\ast,\quad n\in \N.$$
Then $U=\bigoplus_{n=0}^\infty U(n)$.
Observe that for a given level $n\in \N$, the homogeneous space $U(n)\o U(n)$ can be decomposed as
\begin{equation}\mlabel{3.45}
	U(n)\o U(n)=(V_n\o V_n)\oplus (V_n\o W(n)^\ast)\oplus(W(n)^\ast \o V_n)\oplus (W(n)^\ast\o W(n)^\ast).
\end{equation}
Then we have $V_n\o W(n)^\ast \subset U(n)\o U(n)$.

Recall from Definition \mref{df3.1} that a linear map $T:W\ra V$ is called level-preserving if it satisfies $T(W(n))\ssq V_n$ for $n\in \N$. We let $\Hom_{\mathrm{LP}}(W,V)$ denote the space of level-preserving linear maps from $W$ to $V$. For $n\in \N$, let $\{v^{n}_1,v^n_2,\dots,v^n_{p_n}\}$ be a basis of $W(n)$, and let $\{(v^{n}_1)^{\ast},(v^n_2)^\ast,\dots,(v^n_{p_n})^\ast\}$ be the dual basis of $W(n)^{\ast}$. Similar to \meqref{3.22}, there is a natural isomorphism of vector spaces
$$
\Hom_{\mathrm{LP}}(W,V)=\prod_{t=0}^\infty \Hom(W(t),V_t)\cong \prod_{t=0}^{\infty} V_{t}\otimes W(t)^{\ast}\subset \prod_{t=0}^\infty U(t)\mhat{\o}U(t)=D(U\mhat{\o}U),
$$
where $T\in \Hom_{\mathrm{LP}}(W,V)$ corresponds to $\sum_{t=0}^{\infty}\sum_{i=1}^{p_t} T(v_i^t)\otimes (v_i^t)^\ast.$

From now on, we fix a level-preserving linear map $T:W\ra V$.
By Eq.~\meqref{3.45}, we have $T=\sum_{t=0}^{\infty}\sum_{i=1}^{p_t} T(v_i^t)\otimes (v_i^t)^\ast\in V_t\o W(t)^\ast$. Hence we have
 $\sigma(T)=\sum_{t=0}^{\infty}\sum_{i=1}^{p_t} (v_i^t)^\ast\o T(v_i^t)\in W(t)^\ast\o V_t $. Therefore we take the skewsymmetrization
\begin{equation}\mlabel{3.13}
r:=r_T:=T-\sigma(T)= \sum_{t=0}^{\infty}\sum_{i=1}^{p_t} T(v_i^t)\otimes (v_i^t)^\ast- (v_i^t)^\ast\otimes T(v_i^t)\in \SD((V\rtimes W')\mhat{\o} (V\rtimes W')).
\end{equation}

To determine the conditions on $T$ so that $r_T$ is a solution of the $m$-VOYBE, we need the following preparations.
\begin{lm}\mlabel{lm3.11}
	With the setting as above, for homogeneous $a\in V$, we have the equalities
	\begin{align}
	\sum_{i} T(v_i^t)\otimes a'_m (v_i^t)^\ast&=\sum_j T(a_m v_j^{t+m+1-\wt (a)})\otimes (v_j^{t+m+1-\wt (a)})^{\ast}, \mlabel{3.30} \\
	\sum_{i} a'^{\mathrm{op}}_m (v_i^t)^\ast\otimes T(v^t_i)&=\sum_j (v_j^{t+m+1-\wt (a)})^\ast\otimes T((v_j^{t+m+1-\wt (a)})(m)a), \mlabel{3.31}\\
	\sum_{k} a_m (v^s_k)^\ast \o T(v^s_k)&=\sum_{j} (v_j^{\wt (a)-m-1+s})^\ast \o T(a'_m v^{\wt (a)-m-1+s}_j), \mlabel{3.32}\\
	\sum_{k} (v^s_k)^\ast (m) a\o T(v^s_k)&=\sum_{j} (v_j^{\wt (a)-m-1+s})^\ast \o T(a'^{\mathrm{op}}_{m} v_j^{\wt (a)-m-1+s}), \quad t,s\in \N. \mlabel{3.33}
	\end{align}
\end{lm}
\begin{proof}
Note that $a'_m (v_i^t)^\ast\in W(t+m+1-\wt (a))^{\ast}$. Then by \meqref{3.51}, \meqref{3.52} and Definition \mref{df2.3}, we have
\begin{align*}
&\sum_{i} T(v_i^t)\otimes a'_m (v_i^t)^\ast\\
&=  \sum_{i} \sum_jT(v_i^t)\otimes \<a'_m (v_i^t)^\ast,v_j^{t+m+1-\wt (a)}\> (v_j^{t+m+1-\wt (a)})^\ast\\
&=\sum_{i} \sum_j T(v_{i}^t)\otimes\Res_{z}z^m \<Y_{W'}(e^{zL(1)}(-z^{-2})^{L(0)}a,z^{-1})(v_i^t)^\ast, v_j^{t+m+1-\wt (a)}\>(v_j^{t+m+1-\wt (a)})^\ast\\
&=\sum_{i} \sum_jT(v_i^t)\otimes \Res_{z} z^m\<(v_i^t)^\ast, Y_{W}(a,z)v_j^{t+m+1-\wt (a)}\>(v_j^{t+m+1-\wt (a)})^\ast\\
&=\sum_i \sum_jT(\<(v_i^t)^\ast,a_m v_j^{t+m+1-\wt (a)}\>v_i^t)\otimes (v_j^{t+m+1-\wt (a)})^\ast\\
&=\sum_j T(a_m v_j^{t+m+1-\wt (a)})\otimes (v_j^{t+m+1-\wt (a)})^\ast.
\end{align*}
This proves \meqref{3.30}.  Likewise, $a'^{\mathrm{op}}_m (v_i^t)^\ast\in W(t+m+1-\wt (a))^{\ast}$ gives
\begin{align*}
&   \sum_{i} a'^{\mathrm{op}}_m (v_i^t)^\ast\otimes T(v^t_i)\\
	&=\sum_i\sum_j \<a'^{\mathrm{op}}_m (v_i^t)^\ast, v_j^{t+m+1-\wt (a)}\> (v_j^{t+m+1-\wt (a)})^\ast\otimes T(v^t_i)\\
	 &=\sum_{i,j}\Res_{z}z^m\<Y_{W'}(e^{-zL(1)}(-z^{-2})^{L(0)}a,-z^{-1})e^{zL(1)}(v_i^t)^\ast,  v_j^{t+m+1-\wt (a)}\> (v_j^{t+m+1-\wt (a)})^\ast\otimes T(v^t_i)\\
	&=\sum_{i,j}\Res_{z}z^m \<(v_i^t)^\ast,e^{zL(-1)}Y_W(a,-z)v_j^{t+m+1-\wt (a)}\>(v_j^{t+m+1-\wt (a)})^\ast\otimes T(v^t_i)\\
	&=\sum_{i,j} \Res_{z}z^{m} \<(v_i^t)^\ast, Y_{WV}^{W}(v_j^{t+m+1-\wt (a)},z)a\>(v_j^{t+m+1-\wt (a)})^\ast\otimes T(v^t_i)\\
	&=\sum_{i,j}  (v_j^{t+m+1-\wt (a)})^\ast\otimes T(\<(v_i^t)^\ast, (v_j^{t+m+1-\wt (a)})(m)a\>v^t_i)\\
	&=\sum_j (v_j^{t+m+1-\wt (a)})^\ast\otimes T((v_j^{t+m+1-\wt (a)})(m)a).
	\end{align*}
	This proves \meqref{3.31}. The proofs of \meqref{3.32} and \meqref{3.33} are similar, just observing that $a_m(v^s_k)^\ast$ and $(v^s_k)^\ast(m)a$ are both contained in $W(\wt (a)-m-1+s)^\ast$.
\end{proof}

\begin{lm}\mlabel{lm3.10}
		With the setting as above, for any homogeneous $a\in V$, we have
		\begin{equation}\mlabel{3.34}
		\begin{aligned}
	&\sum_{l} ((v^q_l)^\ast)'(m)a\o T(v^q_l)\\
	=&\sum_{j}\Res_{z} z^{-m-2}(-1)^{m+1}  (v_j^{\wt (a)+m+1-q})^\ast \o  T\left( Y_{WV}^W(e^{z^{-1}L(1)}(-z^2)^{L(0)}v_j^{\wt (a)+m+1-q},z)e^{-z^{-1}L(1)}a\right),
		\end{aligned}
		\end{equation}
where $((v^q_l)^\ast)'(m)a=\Res_{z} z^m Y_{W'V}^{W'}(e^{zL(1)}(-z^{-2})^{L(0)}(v^q_l)^\ast,z^{-1})a$, and
\begin{equation}\mlabel{3.35}
		\begin{aligned}
	&	\sum_{l} T(v^q_l)\o ((v^q_l)^\ast)'^{\mathrm{op}}(m) a\\
=&\sum_{j} \Res_{z} z^{-m-2}(-1)^{m+1}(v_j^{\wt (a)+m+1-q})^\ast\o  T\left( Y_{WV}^W(e^{-z^{-1}L(1)}(-z^2)^{L(0)}v_j^{\wt (a)+m+1-q},-z)a\right),
		\end{aligned}
		\end{equation}
where $((v^q_l)^\ast)'^{\mathrm{op}}(m) a=\Res_{z}z^m Y_{W'V}^{W'}(e^{-zL(1)}(-z^{-2})^{L(0)}(v^q_l)^\ast, -z^{-1})e^{zL(1)}a$.
	\end{lm}
\begin{proof}
We prove \meqref{3.34} first. Since the conformal weights of $W$ and $W$ are $0$, we have
	\begin{align*}(-z^{2})^{L(0)} (v_j^{\wt (a)+m+1-q})&=(-1)^{\wt (a)+m+1-q} z^{2\wt (a)+2m+2-2q} v^{\wt (a)+m+1-q}_j,\\
	(-z^{-2})^{L(0)}(v^q_l)^\ast&=(-1)^{q} z^{-2q}(v^q_l)^\ast.
	\end{align*}
Moreover, $((v^q_l)^\ast)'(m)a\in W(\wt (a)+m+1-q)^\ast$ for $q\geq 0$. Then we have
{\small	\begin{align*}
	&\sum_{l} ((v^q_l)^\ast)'(m)a\o T(v^q_l)=\sum_{j,l}\<((v^q_l)^\ast)'(m)a, v_j^{m+1-q+\wt (a)}\> (v_j^{m+1-q+\wt (a)})^\ast \o T(v^q_l)\\
	&=\sum_{j,l}\Res_{z} z^m (v_j^{\wt (a)+m+1-q})^\ast\o \<Y_{W'V}^{W'}(e^{zL(1)}(-z^{-2})^{L(0)}(v^q_l)^\ast,z^{-1})a,v_j^{\wt (a)+m+1-q}\> T(v^q_l)\\
	&=\sum_{j,l}\Res_{z} z^m(v_j^{\wt (a)+m+1-q})^\ast\o \<e^{z^{-1}L(-1)}Y_{W'}(a,-z^{-1})e^{zL(1)}(-z^{-2})^{L(0)}(v^q_l)^\ast, v_j^{\wt (a)+m+1-q}\> T(v^q_l)\\
	&=\sum_{j,l}\Res_{z} z^m (v_j^{\wt (a)+m+1-q})^\ast\o \<(-z^{-2})^{L(0)}(v^q_l)^\ast,   e^{zL(-1)}Y_W(e^{-z^{-1}L(1)}(-z^{2})^{L(0)}a,-z)e^{z^{-1}L(1)}v_j^{\wt (a)+m+1-q} \> T(v^q_l)\\
	&=\sum_{j,l}\Res_{z} z^m(v_j^{\wt (a)+m+1-q})^\ast \o\< (-1)^{q} z^{-2q}(v^q_l)^\ast, Y_{WV}^W(e^{z^{-1}L(1)}v^{\wt (a)+m+1-q}_{j},z) e^{-z^{-1}L(1)}(-1)^{\wt (a)} z^{2\wt (a)}a \> T(v^q_l)\\
	&=\sum_{j,l} \Res_{z} z^{-m-2} (-1)^{m+1}(v_j^{\wt (a)+m+1-q})^\ast \o \<(v^q_l)^\ast,  Y_{WV}^W(e^{z^{-1}L(1)}(-z^2)^{L(0)}v_j^{\wt (a)+m+1-q},z) e^{-z^{-1}L(1)}a\>T(v^q_l)\\
	&=\sum_{j}\Res_{z} z^{-m-2}(-1)^{m+1}  (v_j^{\wt (a)+m+1-q})^\ast \o  T\left( Y_{WV}^W(e^{z^{-1}L(1)}(-z^2)^{L(0)}v_j^{\wt (a)+m+1-q},z)e^{-z^{-1}L(1)}a\right).
	\end{align*}
}
This proves \meqref{3.34}. The proof of \meqref{3.35} is similar.
\end{proof}

\subsection{Strong relative RBOs and solutions of VOYEB in semidirect products}
The level-preserving linear map $T:W\ra V$ has the coadjoint map
\begin{equation}\mlabel{3.36}
T^\ast: V'\ra W',\quad \<T^\ast(f),u\>:=\<f,T(u)\>,
\end{equation}
for any homogeneous $f\in V'$ and $u\in W$. Clearly $T^\ast$ is also a level-preserving map between the contragredient modules $W'$ and $V'$.

We define an intertwining operator $\Y_{WW'}^{V'}$ (see Section 5.4 in \mcite{FHL}) as follows. For $u\in W$, $u^\ast\in W'$, and $a\in V$, let
\begin{equation}\mlabel{3.37}
\<\Y_{WW'}^{V'}(u,z)u^\ast,a\>:=\<u^\ast,Y_{WV}^W(e^{zL(1)}(-z^{-2})^{L(0)}u,z^{-1})a\>.
\end{equation}
 Then we define another intertwining operator $\Y_{W'W}^{V'}$ by the skewsymmetry formula
\begin{equation}\mlabel{3.38}
\Y_{W'W}^{V'}(u^\ast,z)u:=e^{zL(-1)} \Y_{WW'}^{V'}(u,-z)u^\ast.
\end{equation}
Observe that $\Y_{WW'}^{V'}$ and $\Y_{W'W}^{V'}$ are unique up to a scalar if $W$ is an irreducible $V$-module.

We let $Y_{VW'}^{W'}=Y_{W'}$ and $Y_{VV'}^{V'}=Y_{V'}$ be the contragredient module operators with respect to $W'$ and $V'$, respectively. Recall that they are defined by the similar adjoint formulas with respect to $Y_{W}$ and $Y_V$, respectively (see (5.5.4) in \mcite{FHL}). Then we define $Y_{W'V}^{W'}$ by \meqref{3.44'}, and $Y_{V'V}^{V'}$ by the skewsymmetry formula \meqref{2.12} with respect to $Y_{V'}$:
\begin{equation}
Y_{V'V}^{V'}(f,z)a=e^{zL(-1)}  Y_{V'}(a,-z)f,\quad a\in V,f\in V'.
\end{equation}
With the notations as above, we give the following notions.

\begin{df}\mlabel{df3.11}
	Let $(V,Y,\vac,\om)$ be a VOA, and $(W,Y_W)$ be an ordinary $V$-module of conformal weight $\la=0$. Let $T\in \Hom_{\mathrm{LP}}(W,V)$, and let $T^\ast:V'\ra W'$ be given by \meqref{3.36}.
	\begin{enumerate}
\item Let $m\in \Z$. $T$ is called an {\bf $m$-strong relative Rota-Baxter operator} if $T$ is a level-preserving relative RBO, and $T$ and $T^\ast$ satisfy the following compatibility axioms for any $u\in W$ and $f\in V'$.
		\begin{align}
	&	\Res_{z} z^m\left(Y_{W'}(T(u),z)T^\ast(f)-T^\ast(Y_{VV'}^{V'}(T(u),z)f)+T^\ast(\Y_{WW'}^{V'}(u,z)T^\ast(f)) \right)=0, \mlabel{3.39}\\
	& \Res_{z} z^m \left(Y_{W'V}^{W'}(T^\ast(f),z)T(u)-T^\ast(Y_{V'V}^{V'}(f,z)T(u))+T^\ast(\Y_{W'W}^{V'}(T^\ast(f),z)u)\right)=0.\mlabel{3.40}
		\end{align}
	
\item $T$ is called a \name{strong relative Rota-Baxter operator} if $T$ is an $m$-strong relative Rota-Baxter operator for all $m\in \Z$.
		\end{enumerate}
	\end{df}

\begin{remark}
	 In Definition \mref{df3.11}, if $T^\ast:V'\ra W'$ is commutative with the operator $L(-1)$, then \meqref{3.40} follows from \meqref{3.39} since $Y_{W'V}^{W'}$, $Y_{V'V}^{V'}$, and $Y_{W'W}^{V'}$ are defined by the skewsymmetry formulas.
	\end{remark}

The following theorem shows that a strong relative RBO gives rise to a solution to the VOYBE in the semidirect product VOA.

\begin{thm}\mlabel{thm:main2}
Let $V$ be a VOA, $W$ be an ordinary $V$-module of conformal weight $0$, $U=V\rtimes W'$ be the semidirect product VOA, $T\in \Hom_{\mathrm{LP}}(W,V)$ be a level-preserving linear operator, and $r=r_T$ be $T-\sigma(T)\in U^{\mhat{\o}2}$ as in \meqref{3.13}.
Define $r_{12}$, $r_{13}$, and $r_{23}$ as in \meqref{3.46},\meqref{3.47} and \meqref{3.48}$:$
\begin{align*}
r_{12}:&=\sum_{t=0}^{\infty}\sum_{i=1}^{p_t} T(v^t_i)\otimes (v^t_i)^\ast \otimes I - (v_i^t)^\ast \otimes T(v_i^t)\otimes I,\\
r_{13}:&=\sum_{s=0}^{\infty}\sum_{k=1}^{p_s} T(v^s_k)\otimes I\otimes (v^s_k)^\ast- (v_k^s)^\ast \otimes I \otimes T(v_k^s),\\
r_{23}:&=\sum_{q=0}^{\infty}\sum_{l=1}^{p_r} I\otimes T(v^q_l)\otimes (v^q_l)^\ast - I \otimes  (v_l^q)^\ast \otimes T(v_l^q).
\end{align*}
Let $m\in \Z$. Then $r$ satisfies the $m$-VOYBE \meqref{3.22} in the vertex algebra $U=V\rtimes W'$ if and only if $T:W\ra V$ is an $m$-strong relative RBO.
In particular, $r$ is a solution to the VOYBE in the vertex algebra $U$ if and only if $T:W\ra V$ is a strong relative RBO.
\end{thm}
\begin{proof}
By Definition \mref{df2.3}, \meqref{3.28}, the equalities \meqref{3.32} and \meqref{3.33} in Lemma \mref{lm3.11}, and the fact that $(v^t_i)^\ast\cdot_m (v^s_k)^\ast=0$ in $U$ for all $t,s,i,k$, we have
	\begin{align*}
	&r_{12}\cdot_m r_{13}\\
	&= \sum_{s,t\geq 0} \sum_{i,k} \big(T(v^s_k)_m T(v^t_i)\otimes (v^t_i)^\ast\otimes (v^s_k)^\ast- T(v^s_k)_m (v^t_i)^\ast\otimes T(v^t_i)\otimes (v^s_k)^\ast\\
	&\quad -(v^s_k)^\ast(m)T(v^t_i)\otimes (v^t_i)^\ast\otimes T(v^s_k)\big)\\
	&=\underbrace{\sum_{s,t\geq 0}\sum_{i,k}  {T(v^s_k)_m T(v^t_i)\otimes (v^t_i)^\ast\otimes (v^s_k)^\ast}}_{(A1)}
	\underbrace{-\sum_{s,t\geq 0} \sum_{k,j} {(v_j^{t-m-1+s})^\ast\otimes T(T(v^s_k)'_mv_j^{t-m-1+s})\otimes (v^s_k)^\ast}}_{(A2)}\\
	&\quad \underbrace{- \sum_{s,t\geq 0} \sum_{i,j} {(v_j^{s-m-1+t})^\ast\otimes (v_i^t)^\ast \otimes T(T(v^t_i)'^{\mathrm{op}}_mv_j^{s-m-1+t})}}_{(A3)}\\
	&=(A1)+(A2)+(A3).
	\end{align*}
	By Definition \mref{df2.3}, \meqref{3.34} in Lemma \mref{lm3.10}, \meqref{3.30} in Lemma \mref{lm3.11}, and the fact that $((v_l^q))^\ast \cdot'_m (v^t_i)^\ast =0$ in $U$ for $q,t,i,l$, we have
\begin{align*}
&-r_{23}\cdot'_mr_{12}\\
	&=\sum_{q,t\geq 0} \sum_{i,l}\big(-T(v^t_i)\otimes T(v^q_l)'_m(v^t_i)^\ast  \otimes (v^q_l)^\ast +(v^t_i)^\ast \otimes T(v^q_l)'_m T(v^t_i) \otimes (v^q_l)^\ast\\
	&\ \ - (v^t_i)^\ast\otimes ((v^q_l)^\ast)'(m) T(v^t_i) \otimes T(v^q_l)\big)\\
	&=\underbrace{-\sum_{q,t\geq 0} \sum_{j,l} {T(T(v^q_l)_m v^{t+m+1-q}_j) \otimes (v^{t+m+1-q}_j)^\ast\otimes (v^q_l)^\ast}}_{(B1)} +\underbrace{\sum_{q,t\geq 0} \sum_{i,l} {(v^t_i)^\ast \otimes T(v^q_l)'_m T(v^t_i) \otimes (v^q_l)^\ast}}_{(B2)}\\
	&\ \ \underbrace{- \sum_{q,t\geq 0} \sum_{i,j} {\Res_{z}(-1)^{m+1} z^{-m-2} (v^t_i)^\ast\otimes (v^{t+m+1-q}_j)^\ast \o  T\left(Y_{WV}^W(e^{z^{-1}L(1)}(-z^2)^{L(0)}v_j^{t+m+1-q},z)e^{-z^{-1}L(1)}T(v^t_i)\right)}\bigg)}_{(B3)}\\&= (B1)+(B2)+(B3).
	\end{align*}
	Finally, by Definition \mref{df2.3}, \meqref{3.35} in Lemma \mref{lm3.10}, \meqref{3.31} in Lemma \meqref{lm3.11}, and the fact that $(v^s_k)^\ast\cdot'^{\mathrm{op}}_m (v^q_l)^\ast=0$ for $q,s,k,l$, we have
	\begin{align*}
	&r_{13}\cdot'^{\mathrm{op}}_m r_{23}\\
	&= \sum_{q,s\geq 0} \sum_{k,l}\big(-T(v^s_k)\otimes (v^q_l)^\ast\otimes T(v^q_l)'^{\mathrm{op}}_m(v^s_k)^\ast -(v^s_k)^\ast\otimes T(v^q_l)\otimes ((v^q_l)^\ast)'^{\mathrm{op}}(m) T(v^s_k) \\
	&\quad +(v^s_k)^\ast\otimes (v^q_l)^\ast\otimes T(v^q_l)'^{\mathrm{op}}_m T(v^s_k)\big)\\
	&= \underbrace{-\sum_{q,s\geq 0} \sum_{j,l} {T((v^{s+m+1-q}_j)(m)T(v^q_l))\otimes (v^q_l)^\ast \otimes (v^{s+m+1-q}_j)^\ast}}_{(C1)} \\	
	&\quad \underbrace{-\sum_{q,s\geq 0} \sum_{k,j} \bigg(\Res_{z} (-1)^{m+1}z^{-m-2}(v^s_k)^\ast\o {T\left(Y_{WV}^W(e^{-z^{-1}L(1)}(-z^2)^{L(0)}v_j^{s+m+1-q},-z)T(v^s_k)\right)\otimes (v^{s+m+1-q}_j)^\ast}\bigg)}_{(C2)}\\
	&\quad +\underbrace{\sum_{q,s\geq 0} \sum_{k,l} {(v^s_k)^\ast\otimes (v^q_l)^\ast\otimes T(v^q_l)'^{\mathrm{op}}_m T(v^s_k)}}_{(C3)}\\
	&=(C1)+(C2)+(C3).
	\end{align*}
	Then we have
$$r_{12}\cdot_m r_{13}-r_{23}\cdot'_m r_{12}+ r_{13}\cdot'^{\mathrm{op}}_mr_{23}=(A1)+(A2)+(A3) +(B1)+(B2)+(B3)+(C1)+(C2)+(C3).$$
Note that
	\begin{align*}
	&(A1)+(B1)+(C1)\\
	&=\sum_{s,t\geq 0} \sum_{k,i} T(v^s_k)_m T(v^t_i)\otimes (v^t_i)^\ast\otimes (v^s_k)^\ast
- 	\sum_{q,t\geq 0} \sum_{l,j} T(T(v^q_l)_m v^{t+m+1-q}_j) \otimes (v^{t+m+1-q}_j)^\ast\otimes (v^q_l)^\ast \\
	&\ \ \  -\sum_{q,s\geq 0}\sum_{j,l} T((v^{s+m+1-q}_j)(m)T(v^q_l))\otimes (v^q_l)^\ast \otimes (v^{s+m+1-q}_j)^\ast.
	\end{align*}
Changing the variables $(q,t+m+1-q)\mapsto (s,t)$ in $(B1)$ and $(s+m+1-q,q)\mapsto (s,t)$ in $(C1)$
as in the proof of Lemma \mref{lm3.5}, and observing that $T(v^s_k)_m T(v^t_i)\in V_{s-m-1+t}=0$ for $s,k$ if $s+t<m+1$, we obtain
\begin{align*}
&(A1)+(B1)+(C1)\\
&=\sum_{s,t\geq 0,s+t\geq m+1} \sum_{k,i} T(v^s_k)_m T(v^t_i)\otimes (v^t_i)^\ast\otimes (v^s_k)^\ast  -\sum_{s,t\geq 0, s+t\geq m+1} \sum_{k,i} T(T(v^s_k)_m v^t_i)\otimes (v^t_i)^\ast\otimes (v^s_k)^\ast\\
&\quad  -\sum_{s,t\geq 0,s+t\geq m+1} \sum_{k,i} T((v^s_k)(m)T(v^t_i))\otimes (v^t_i)^\ast\otimes (v^s_k)^\ast\\
&=\sum_{s,t\geq 0,s+t\geq m+1} \sum_{k,i} \left(T(v^s_k)_m T(v^t_i)- T(T(v^s_k)_m v^t_i)-T((v^s_k)(m)T(v^t_i))\right)\otimes (v^t_i)^\ast\otimes (v^s_k)^\ast,
\end{align*}
which is contained in the subspace $ \prod_{s,t\geq 0,s+t\geq m+1} V_{s+t-m-1}\otimes W(t)^\ast\otimes W(s)^\ast\subset U^{\mhat{\o}3}$. Since $(v^t_i)^\ast$ and $(v^s_k)^\ast$ are bases of $W(t)^\ast$ and $W(s)^\ast$ respectively, we have
$$(A1)+(B1)+(C1)=0 \iff T(v^s_k)_m T(v^t_i)- T(T(v^s_k)_m v^t_i)-T((v^s_k)(m)T(v^t_i))=0,$$
for $s,t\in \N,$ and $ 1\leq k\leq p_s,1\leq i\leq p_t$. Moreover, since $(v^t_i)$ is a basis of $W(t)$ for each $t\geq 0$, it follows that
\begin{equation}\mlabel{3.41}
(A1)+(B1)+(C1)=0 \iff  T(u)_m T(v)=T(T(u)_mv)+T(u(m)T(v)), \quad u, v\in W,
\end{equation}
that is, $T:W\ra V$ is an $m$-relative RBO.
	
We perform a similar analysis for the other terms in $r_{12}\cdot_m r_{13}-r_{23}\cdot'_m r_{12}+ r_{13}\cdot'^{\mathrm{op}}_mr_{23}$. Note that
	\begin{align*}
&	(A2)+(B2)+(C2)\\
&=-\sum_{s,t\geq 0}\sum_{k,j} (v_j^{t-m-1+s})^\ast\otimes T(T(v^s_k)'_mv_j^{t-m-1+s})\otimes (v^s_k)^\ast+\sum_{q,t\geq 0} \sum_{i,l} (v^t_i)^\ast \otimes T(v^q_l)'_m T(v^t_i) \otimes (v^q_l)^\ast\\
	&\ \ \ -\sum_{q,s\geq 0} \sum_{k,j} \Res_{z} \left((-1)^{m+1}z^{-m-2}(v^s_k)^\ast\o  T\left(Y_{WV}^W(e^{-z^{-1}L(1)}(-z^2)^{L(0)}v_j^{s+m+1-q},-z)T(v^s_k)\right)\otimes (v^{s+m+1-q}_j)^\ast\right).
	\end{align*}
Changing the variables $(t-m-1+s,s)\mapsto (t,q)$ in $(A2)$ and $(s,s+m+1-q)\mapsto (t,q)$ in $(C2)$, and observing that $T(v^q_l)'_m T(v^t_i)\in V_{t+m+1-q}=0$ if $q-t>m+1$, we obtain
	\begin{align*}
	&(A2)+(B2)+(C2)\\
	&=-\sum_{q,t\geq 0,r-t\leq m+1} \sum_{l,i} (v_i^{t})^\ast\otimes T(T(v^q_l)'_mv^t_i)\otimes (v^q_l)^\ast + \sum_{t,r\geq 0,q-t\leq m+1} \sum_{l,i} (v^t_i)^\ast \otimes T(v^q_l)'_m T(v^t_i)\otimes (v^q_l)^\ast\\
	&\ \ \ -\sum_{q,t\geq 0,q-t\leq m+1} \sum_{l,i} \left(\Res_{z} (-1)^{m+1}z^{-m-2}(v^t_i)^\ast\o T\left(Y_{WV}^W(e^{-z^{-1}L(1)}(-z^2)^{L(0)}v_l^{q},-z)T(v^t_i)\right)\otimes (v^{q}_l)^\ast\right),
	\end{align*}
which is contained in the subspace $ \prod_{t,q\geq 0,q-t\leq m+1} W(t)^\ast\otimes V_{t+q-m-1}\otimes W(q)^\ast\subset U^{\mhat{\o}3}$. Then
$(A2)+(B2)+(C2)=0$ if and only if $$-T(T(v^q_l)'_mv^t_i)+T(v^q_l)'_m T(v^t_i)-\Res_{z} (-1)^{m+1} z^{-m-2}\big( T(Y_{WV}^W(e^{-z^{-1}L(1)}(-z^2)^{L(0)}v_l^{r},-z)T(v^t_i))\big)=0$$
in $V_{t+q-m-1}$, for $q,t\in \N$ and $1\leq l\leq p_q$, $1\leq i\leq p_t$.
	
Let $f\in V_{t+m+1-r}^\ast\subset V'$. We apply $\<-,f\>$ to the terms on the left hand side.  Then by \meqref{3.51}, \meqref{3.36} and \meqref{3.37}, we have
\begin{align*}
\<T(T(v^q_l)'_m v^t_i),f\>&=\<T(v^q_l)'_m v^t_i, T^\ast(f)\>=\Res_{z} z^m \<Y_W(e^{zL(1)}(-z^{-2})^{L(0)}T(v^q_l),z^{-1})v^t_i, T^\ast(f)\>\\
&=\Res_{z} z^m \<v^t_i, Y_{W'}(T(v^q_l),z)T^\ast(f) \>,\\
\<T(v^q_l)'_m T(v^t_i), f\>&=\<Y_V(e^{zL(1)}(-z^{-2})^{L(0)}T(v^q_l),z^{-1})T(v^t_i),f\>=\<T(v^t_i),Y_{VV'}^{V'}(T(v^q_l),z)f\>\\
&=\<v^t_i, T^\ast(Y_{VV'}^{V'}(T(v^q_l),z)f)\>.
\end{align*}
We write $\Y_{WW'}^{V'}(v,z)=\sum_{n\in \Z} v[n]z^{-n-1}$. Then by \meqref{3.37} we have
\begin{align*}
&\Res_{z} (-1)^{m+1} z^{-m-2} \< T(Y_{WV}^W(e^{-z^{-1}L(1)}(-z^2)^{L(0)}v_l^q,-z)T(v^t_i)),f\>\\
&=\Res_{z} (-1)^{m+1} z^{-m-2} \<T(v^t_i), \Y_{WW'}^{V'}(v^q_l,-z^{-1})T^\ast(f)\>\\
&=\Res_{z} (-1)^{m+1} z^{-m-2} \left\<T(v^t_i),\sum_{n\in \Z}v^q_l[n] T^\ast(f)  \right\>(-1)^{n+1} z^{n+1}=\<T(v^t_i), v^q_l[m]T^\ast(f)\>\\
&=\Res_{z} z^m \<v_i^t,T^\ast(\Y_{WW'}^{V'}(v^q_l,z)T^\ast(f))\>.
\end{align*}
Since an element $\al\in V_{t+q-m-1}$ is $0$ if and only if $\<\al,f\>=0$ for all $f\in V_{t+q-m-1}^\ast$, we have
\begin{align*}
&(A2)+(B2)+(C2)=0 \iff\\
&	\Res_{z} z^m\big(\<v^t_i, -Y_{W'}(T(v^q_l),z)T^\ast(f)+T^\ast(Y_{VV'}^{V'}(T(v^q_l),z)f)-T^\ast(\Y_{WW'}^{V'}(v^q_l,z)T^\ast(f)) \>\big)=0,
\end{align*}
	for all $q,t\in \N$, $1\leq l\leq p_q$, $1\leq i\leq p_t$, and $f\in V_{t+q-m-1}^\ast$. But $(v^t_i)$ is a basis of $W(t)$. So we have
	\begin{equation}\label{3.42}
	\begin{aligned}
&	(A2)+(B2)+(C2)=0 \iff \\
&	\Res_{z} z^m\left(Y_{W'}(T(u),z)T^\ast(f)-T^\ast(Y_{VV'}^{V'}(T(u),z)f)+T^\ast(\Y_{WW'}^{V'}(u,z)T^\ast(f)) \right)=0,	u\in W, f\in V',
	\end{aligned}
	\end{equation}
that is, $T$ and $T^\ast$ satisfy \meqref{3.39}.
	
Finally, we examine the sum $(A3)+(B3)+(C3)$. Note that
	\begin{align*}
&(A3)+(B3)+(C3)\\
&=-\sum_{s,t\geq 0}\sum_{i,j} (v_j^{s-m-1+t})^\ast\otimes (v_i^t)^\ast \otimes T(T(v^t_i)'^{\mathrm{op}}_mv_j^{s-m-1+t})\\
&\ \ - \sum_{q,t\geq 0}\sum_{i,j} \left(\Res_{z}(-1)^{m+1} z^{-m-2} (v^t_i)^\ast\otimes (v^{t+m+1-q}_j)^\ast \o  T\left(Y_{WV}^W(e^{z^{-1}L(1)}(-z^2)^{L(0)}v_j^{t+m+1-q},z)e^{-z^{-1}L(1)}T(v^t_i)\right)\right)\\
	&\ \ +\sum_{q,s\geq 0} \sum_{k,l} (v^s_k)^\ast\otimes (v^q_l)^\ast\otimes T(v^q_l)'^{\mathrm{op}}_m T(v^s_k).
	\end{align*}
	Changing the variables $(s-m-1+t,t)\mapsto (s,q)$ in $(A3)$ and $(t,t+m+1-q)\mapsto (s,q)$ in $(C3)$, and observing that $T(v^q_l)'^{\mathrm{op}}_m T(v^s_k)\in V_{s+m+1-q}=0$ if $q-s>m+1$, we have
		\begin{align*}
	&(A3)+(B3)+(C3)\\
	&=-\sum_{q,s\geq 0,q-s\leq m+1}\sum_{k,l} (v_k^{s})^\ast\otimes (v_l^q)^\ast \otimes T(T(v^q_l)'^{\mathrm{op}}_mv_k^{s})\\
	&\ \ - \sum_{q,s\geq 0,q-s\leq m+1}^\infty\sum_{k,l} \left(\Res_{z}(-1)^{m+1} z^{-m-2} (v^s_k)^\ast\otimes (v^q_l)^\ast \o T\left(Y_{WV}^W(e^{z^{-1}L(1)}(-z^2)^{L(0)}v_l^q,z)e^{-z^{-1}L(1)}T(v^s_k)\right)\right)\\
	&\ \ +\sum_{q,s\geq 0,r-s\leq m+1} \sum_{k,l} (v^s_k)^\ast\otimes (v^q_l)^\ast\otimes T(v^q_l)'^{\mathrm{op}}_m T(v^s_k),
	\end{align*}
	which is contained in the subspace $\prod_{q,s\geq 0,q-s\leq m+1} W(s)^\ast\otimes W(q)^\ast\otimes V_{q+s-m-1}\subset U^{\mhat{\o}3}.$ Then
$(A3)+(B3)+(C3)=0$ if and only if $$-T(T(v^q_l)'^{\mathrm{op}}_mv_k^{s})+ T(v^q_l)'^{\mathrm{op}}_m T(v^s_k)-\Res_{z}(-1)^{m+1} z^{-m-2}\big(
T(Y_{WV}^W(e^{z^{-1}L(1)}(-z^2)^{L(0)}v_l^q,z)e^{-z^{-1}L(1)}T(v^s_k)\big)=0$$ in $V_{q+s-m-1}.$
Now let $f\in V^\ast_{q+s-m-1}$, and apply $\<-,f\>$ to the terms on the right hand side. By \meqref{3.52} and \meqref{3.36}, we have
	\begin{align*}
	\<T(T(v^q_l)'^{\mathrm{op}}_mv_k^{s}), f\>&=\<T(v^q_l)'^{\mathrm{op}}_m v_k^s, T^\ast(f)\>=\Res_{z} z^m \<Y_W(e^{-zL(1)}(-z^{-2})^{L(0)}T(v^q_l),-z^{-1})e^{zL(1)}v_k^s, T^\ast(f)\>\\
	&=\Res_{z} z^m \<v^s_k,e^{zL(-1)}Y_{W'}(T(v^q_l),-z)T^\ast(f)\>
	=\Res_{z} z^m \<v^s_k, Y_{W'V}^{W'}(T^\ast(f),z)T(v^q_l)\>.\\
	\<T(v^q_l)'^{\mathrm{op}}_m T(v^s_k),f\>&=\Res_{z} z^m \<Y_V(e^{-zL(1)}(-z^{-2})^{L(0)}T(v^q_l),-z^{-1})e^{zL(1)}T(v^s_k), f\>\\
	&=\Res_{z} z^m \<v^s_k, T^\ast(e^{zL(-1)}Y_{VV'}^{V'}(T(v^q_l),-z)f)\>
	=\Res_{z} z^m\<v^s_k, T^\ast(Y_{V'V}^{V'}(f,z)T(v^q_l))\>.
	\end{align*}
Write $\Y_{W'W}^{V'}(\al, z)=\sum_{n\in \Z} \al\{n\}z^{-n-1} $. Then by \meqref{3.37} and \meqref{3.38},  we have
\begin{align*}
&\Res_{z}(-1)^{m+1} z^{-m-2}\<T(Y_{WV}^W(e^{z^{-1}L(1)}(-z^2)^{L(0)}v_l^{r},z)e^{-z^{-1}L(1)}T(v^s_k),f\>\\
&=\Res_{z} (-1)^{m+1}z^{-m-2} \<e^{-z^{-1}L(1)}T(v^s_k), \Y_{WW'}^{V'}(v^q_l, z^{-1})T^\ast(f)\>\\
&=\Res_{z} (-1)^{m+1} z^{-m-2} \<v^s_k, T^\ast( \Y_{W'W}^{V'}(T^\ast(f),-z^{-1}) v^q_l)\>=\<v^s_k, T^\ast(T^\ast(f)\{m\}v^q_l)\>\\
&=\Res_{z} z^m \<v^s_k, T^\ast(\Y_{W'W}^{V'}(T^\ast(f),z)v^q_l)\>.
\end{align*}
It follows that $(A3)+(B3)+(C3)=0$ if and only if 
$$\Res_{z}z^m\<v^s_k, -Y_{W'V}^{W'}(T^\ast(f),z)T(v^q_l)+T^\ast(Y_{V'V}^{V'}(f,z)T(v^q_l))-T^\ast(\Y_{W'W}^{V'}(T^\ast(f),z)v^q_l)\>=0,
$$	
for all $q,s\in \N$, $1\leq k\leq p_s$, $1\leq l\leq p_q$, and $f\in V_{q+s-m-1}^\ast$. Then $(v^s_k)$ being a basis of $W(s)^\ast$ yields
\begin{equation}\mlabel{3.43}
	\begin{aligned}
	&(A3)+(B3)+(C3)=0\iff\\
&\Res_{z}z^m\left(Y_{W'V}^{W'}(T^\ast(f),z)T(u))-T^\ast(Y_{V'V}^{V'}(f,z)T(u))+T^\ast(\Y_{W'W}^{V'}(T^\ast(f),z)u )\right)=0,
\end{aligned}
	\end{equation}
for all $f\in V'$ and $u\in W$, that is, $T$ and $T^\ast$ satisfy \meqref{3.40}.

Note that for given $p_1,p_2,p_3\in \N$, the subspaces $V_{p_1}\otimes W(p_1)^\ast\otimes W(p_3)^\ast$, $W(p_1)^\ast\otimes V_{p_2}\otimes W(p_3)^\ast$ and $W(p_1)^\ast\otimes W(p_2)^\ast\otimes V_{p_3}$ are in direct sum within the vector space $U(p_1)\otimes U(p_2)\otimes U(p_3)$. Furthermore, by Definition \mref{df3.5} we have $U^{\mhat{\o}3}=\prod_{p_1,p_2,p_3\geq 0} U(p_1)\o U(p_2)\o U(p_3)$. By our discussion above,
	\begin{align*}
	(A1)+(B1)+(C1)&\in  \prod_{s,t \geq 0,s+t\geq m+1} V_{s+t-m-1}\otimes W(t)^\ast\otimes W(s)^\ast\subset U^{\mhat{\o}3},\\
	(A2)+(B2)+(C2)&\in \prod_{q,t\geq 0,q-t\leq m+1} W(t)^\ast\otimes V_{t+q-m-1}\otimes W(q)^\ast\subset U^{\mhat{\o}3},\\
	(A3)+(B3)+(C3)&\in \prod_{q,s\geq 0,q-s\leq m+1} W(s)^\ast\otimes W(q)^\ast\otimes V_{q+s-m-1}\subset U^{\mhat{\o}3}.
	\end{align*}
Since the three subspaces are in direct sum, 
$r_{12}\cdot_m r_{13}-r_{23}\cdot'_m r_{12}+ r_{13}\cdot'^{\mathrm{op}}_mr_{23}=0$ means
$$(A1)+(B1)+(C1)=(A2)+(B2)+(C2)=(A3)+(B3)+(C3)=0.$$
By \meqref{3.41}, \meqref{3.42} and \meqref{3.43}, we see that $r_{12}\cdot_m r_{13}-r_{23}\cdot'_m r_{12}+ r_{13}\cdot'^{\mathrm{op}}_mr_{23}=0$ if and only if $T:W\ra V$ is an $m$-relative RBO, and $T$ and its coadjoint $T^\ast$ satisfy \meqref{3.39} and \meqref{3.40}. This means that $T:W\ra V$ is a strong $m$-relative RBO. This completes the proof of Theorem~\mref{thm:main2}.
\end{proof}

\begin{remark}
If we want to drop the condition that the conformal weight $\la$ of $W$ is $0$, then $U=V\rtimes W'$ is a $\Q$-graded vertex algebra, and we have to adjust the definitions of $\al \cdot'_m \b$ and $\al\cdot_m'^{\mathrm{op}} \b$ in Definition \mref{df2.3} to accommodate the appearance of the term $z^{-2\la}$ in $Y'_{U}(u^\ast,z)a=Y_{W'V}^{W'}(e^{zL(1)}(-z^{-2})^{L(0)}u^\ast, z^{-1})a$, where $u^\ast\in W'$ and $a\in V$. Although there might still be a way to make things work, this is not our focus of solving the VOYBE. So we make the assumption that $\la=0$ to simplify the discussion.
\mlabel{rk:wt}	
\end{remark}

With notations similar to those in Corollary~\mref{co:2.42}, we obtain an embedding
\begin{equation}\mlabel{3.61}
\mathrm{StrRBO}_{\mathrm{LP}}(W,V)(m)\hookrightarrow \SD_{\mathrm{sol}}((V\rtimes W')\rtimes (V\rtimes W'))(m), \quad m\in \Z.
\end{equation}
Here $\mathrm{StrRBO}_{\mathrm{LP}}(W,V)(m)$ is the set of level-preserving strong $m$-RBO from the $V$-module $W$ to $V$, and the set on the right hand side is the set of skewsymmetric solutions to the $m$-VOYBE in the VOA $V\rtimes W'$.

\subsection{The coadjoint case}

In this subsection, we let $(U,Y_U,\vac,\om)$ be a VOA.
Consider the case when $(W,Y_W)$ is the coadjoint module $(U',Y_{U'})$, and $T$ is a level-preserving map $T:U'\ra U$.

By \meqref{3.36} we have
$$T^\ast: U'\ra (U')'=U,\ \<T^\ast(f),g\>=\<f,T(g)\>,\quad f,g\in U'.$$
In particular, by Definition \mref{df2.16} and \meqref{3.36}, if $T$ is symmetric (resp. skewsymmetric), then we have $T^\ast=T$ (resp. $T^\ast=-T$).
\begin{lm}
The intertwining operator $\Y_{U'U}^{U'}$ given by \meqref{3.37} is the same as the vertex operator $Y_{U'U}^{U'}$ that is given by the skewsymmetry formula with respect to $Y_{U'}$.
	\end{lm}
\begin{proof}
Let $a,b\in U$ and $f\in U'$ be homogeneous elements. Then by \meqref{3.37} and \meqref{3.44'} we have
	\begin{align*}
	 &\<\Y_{U'U}^{U'}(f,z)a,b\>=\<a,Y_{U'U}^{U'}(e^{zL(1)}(-z^{-2})^{L(0)}f,z^{-1})b\>\\
	 &=\<e^{z^{-1}L(1)}a,Y_{U'}(b,-z^{-1})e^{zL(1)}(-z^{-2})^{L(0)}f\>\\
	 &=\<e^{zL(-1)}Y_U(e^{-z^{-1}L(1)}(-z^{2})^{L(0)}b,-z)e^{z^{-1}L(1)}a, (-z^{-2})^{L(0)}f\>\\
	&=(-1)^{\wt (f)+\wt (b)}z^{2\wt (b)-2\wt (f)}\<Y_U(e^{z^{-1}L(1)}a,z)e^{-z^{-1}L(1)}b, f\> \\
	&=(-1)^{\wt (f)+\wt (b)}z^{2\wt (b)-2\wt (f)}\sum_{j\geq 0}\sum_{i\geq 0} \frac{z^{\wt (f)-\wt (b)-\wt (a)}}{i!j!}(-1)^j\<(L(1)^ia)_{\wt (a)+\wt (b)-\wt (f)-i-j-1}(L(1)^jb),f\>\\
	&=\sum_{j\geq 0}\sum_{i\geq 0} \frac{z^{-\wt (a)+\wt (b)-\wt (f)}}{i!j!}(-1)^{\wt (f)+j-\wt (b)}\<(L(1)^ia)_{\wt (a)+\wt (b)-\wt (f)-i-j-1}(L(1)^jb),f\>\\
	&=(-1)^{\wt (a)}z^{-2\wt (a)}\<f,Y_U(e^{-zL(1)}a,-z^{-1})e^{zL(1)}b\>\\
	 &=\<f,Y_U(e^{-zL(1)}(-z^{-2})^{L(0)}a,-z^{-1})e^{zL(1)}b\>\\
	 &=\<Y_{U'}(a,-z)f,e^{zL(1)}b\>=\<Y_{U'U}^{U'}(f,z)a,b\>.
	\end{align*}
	Hence $\Y_{U'U}^{U'}(f,z)a=Y_{U'U}^{U'}(f,z)a$ for $f\in U'$ and $a\in U$, and so $\Y_{U'U}^{U'}=Y_{U'U}^{U'}$.
	\end{proof}
It is also easy to check that the rest of the intertwining operators appearing in \meqref{3.39} and \meqref{3.40}, with $V=U$,  satisfy $Y_{W'}=Y_U$, $Y_{UU'}^{U'}=Y_{U'}$, $Y_{W'U}^{W'}=Y_{U}$, and $\Y_{W'W}^{U'}=Y_{U'}$. In particular, if $T$ is skewsymmetric: $T=-T^\ast$, then both \meqref{3.39} and \meqref{3.40} become
$$\Res_{z} z^m\left(-Y_{U}(T(f),z)T(g)+T(Y_{U'}(T(f),z)g)+T(Y_{U'U}^{U'}(f,z)T(g)) \right)=0,$$
which is the condition that $T$ is an $m$-relative RBO. Hence we have the following conclusion.

\begin{lm}
For $m\in \Z$, any level-preserving skewsymmetric $m$-relative RBO $T:U'\ra U$ is strong.
	\end{lm}

Furthermore, given a symmetric linear map $T:U'\ra U$, we note that the skewsymmetrization $r=T-\sigma(T)$ given by \meqref{3.13} is a nonzero element in $(U\rtimes U)^{\mhat{\o}2}$, due to the fact that the tensor form of $T$ and $\sigma(T)$ are in linearly independent subspaces, in view of \meqref{3.45}. Then by \meqref{3.61}, we have an embedding:
\begin{equation}\mlabel{3.62}
\mathrm{RBO}^{\mathrm{sk}}_{\mathrm{LP}}(U',U)(m)\hookrightarrow \SD_{\mathrm{sol}} ((U\rtimes U)\mhat{\o} (U\rtimes U))(m),
\end{equation}
where $\mathrm{RBO}^{\mathrm{sk}}_{\mathrm{LP}}(U',U)(m)$ is the set of skewsymmetric level-preserving $m$-relative RBOs $T:U'\ra U$. Combining \meqref{3.62} with \meqref{2.42}, we have an embedding
\begin{equation}\label{eq:in}
\SD_{\mathrm{sol}}(U\mhat{\o}U)
\big(\cong \mathrm{RBO}^{\mathrm{sk}}_{\mathrm{LP}}(U',U)(m)\big)\hookrightarrow \SD_{\mathrm{sol}} ((U\rtimes U)\mhat{\o} (U\rtimes U))(m),
\end{equation}
which leads to the following conclusion.
\begin{coro}
	Let $(U,Y_U,\vac,\om)$ be a VOA and $m\in \Z$. Every skewsymmetric solution $r$ of the $m$-VOYBE in $U$ is a skewsymmetric solution to the $m$-VOYBE in $U\rtimes U$, via the embedding \eqref{eq:in}.
	\end{coro}

\section{Relations with the classical Yang-Baxter equation}
\mlabel{s:cybe}
Theorems \mref{thm:main1} and \mref{thm:main2} are generalizations of the classical results about the relationship between relative RBOs (also known as $\O$-operators) and the CYBE for (finite-dimensional) Lie algebras (cf. \mcite{B2}, see also \mcite{K1,S}).

\subsection{Solutions of CYBE and relative RBOs}

We first observe some facts about the first-level Lie algebra of a VOA and its modules.
Let $V$ be a VOA, $W$ be an admissible $V$-module, and $T\in \Hom_{\mathrm{LP}}(W,V)$ be a level-preserving linear operator.
By the Jacobi identities in Definitions \mref{df:va} and \mref{df2.13}, the first level $V_1$ of the VOA $V$ is a Lie algebra with respect to the Lie bracket
$$[a,b]=a_0b, \quad a,b\in V_1,$$
and $W(1)$ is a module over the Lie algebra $V_1$, with respect to
\begin{equation}\mlabel{3.28b}
\rho: V_1\ra \gl(W(1)),\quad  \rho(a)u=a_0u=\Res_{z} Y_W(a,z)u, \quad a\in V_1, u\in W(1).
\end{equation}

Suppose that $V_1$ consists of quasi-primary vectors, that is, $L(1)V_1=0$ (see \mcite{FHL} as well as \mcite{L1}). Then by \meqref{3.9b}, for $u^\ast\in W(1)^\ast$, $v\in W(1)$ and $a\in V_1$, we have
$$\<a_{0}u^\ast,v\>=\Big\<u^\ast,\sum_{j\geq 0}\frac{(-1)^{\wt (a)}}{j!}(L(1)^ja)_{2\wt (a)-j-2}v\Big\>=\<u^\ast, -a_0v\>=-\<u^\ast, \rho(a)v\>.$$
Therefore, the first level $W(1)^\ast$ of the contragredient $V$-module $W'$ is the dual module of the Lie algebra $V_1$-module $W(1)$:
\begin{equation}\mlabel{3.29}
\rho^\ast: V_1\ra \gl(W(1)^\ast),\quad \rho^\ast(a)u^\ast=a_0u^\ast=\Res_{z} Y'_{W}(a,z)u^\ast,\quad
a\in V_1, u^\ast\in W(1)^\ast.
\end{equation}

Let $(U,Y_U,\vac,\om)$ be a VOA such that $U(1)$ consists of quasi-primary vectors and let $r=\sum_{t=0}^\infty r^t=\sum_{t=0}^\infty\sum_{i=1}^{p_t} \al^t_i\otimes \b^t_i-\b^t_i\otimes \al^t_i$ be a diagonal skewsymmetric two-tensor in $D(U\mhat{\o} U)$. Let
 \begin{equation}\mlabel{4.58}
 R:=r^1=\sum_{i=1}^{p_1} \al^1_i\otimes \b^1_i-\b^1_i\otimes \al^1_i\in U(1)\o U(1).
 \end{equation}
Then $R$ is a skewsymmetric two tensor in the Lie algebra $U(1)$.

\begin{lm}\mlabel{lm4.1}
If $r$ is a skewsymmetric solution to the $0$-VOYBE in $U$, then $R$ is a skewsymmetric solution to the CYBE in the Lie algebra $U(1)$.
 	\end{lm}
 \begin{proof}
 	For any $\al,\b\in U(1)$, since $L(1)\al=L(1)\b=0$, by Definition \meqref{df2.3}, we have
 	\begin{equation}\mlabel{4.55}
 	\al\cdot_0 \b=[\b,\al],\quad \al\cdot'_0 \b=(-1)\al_0 \b=[\b,\al],\quad \al\cdot'^{\mathrm{op}}_0 \b=\b_0\al=[\b,\al].
 	\end{equation}
By Lemma \mref{lm3.5}, if $r$ is a solution to the $0$-VOYBE, then taking the projection of $r_{12}\cdot_0 r_{13}-r_{23}\cdot'_0 r_{12}+r_{13}\cdot'^{\mathrm{op}}_0 r_{23}$ onto the homogeneous subspace $U(1)\o U(1)\o U(1)$, we obtain
$$r^1_{12}\cdot_0 r_{13}^1-r_{23}^1\cdot'_0 r_{12}^{1}+r_{13}^{1}\cdot'^{\mathrm{op}}_{0} r_{23}^1=0.$$
By \meqref{4.55} and the definitions in \meqref{2.16}-\meqref{2.18}, we have
$$r^1_{12}\cdot_0 r_{13}^1=[R_{13},R_{12}],\quad -r_{23}^1\cdot'_0 r_{12}^{1}=[R_{23},R_{12}],\quad r_{13}^{1}\cdot'^{\mathrm{op}}_{0} r_{23}^1=[R_{23},R_{13}],$$
where $R_{12},R_{13}$, and $R_{23}$ are elements in the universal enveloping algebra $\mathcal{U}(U(1))$ defined from $R$ in \meqref{4.58} in the conventional way.
Hence we have $[R_{12},R_{13}]+[R_{12},R_{23}]+[R_{13},R_{23}]=0$, which is the CYBE (see \meqref{eq:1}).  	
\end{proof}

Let $T_r:U'\ra U$ be given by \meqref{3.23b}. Consider the restriction map
\begin{equation}\mlabel{4.60}
T_{R}:=T_{r}|_{U(1)^\ast}:U(1)^\ast\ra U(1).
\end{equation}

\begin{lm}\mlabel{lm4.2}
If $T_r: U'\ra U$ is a $0$-relative RBO, then $T_R$ is a relative RBO of the Lie algebra $U(1)$ with respect to the module $U(1)^\ast$.
\end{lm}
\begin{proof}
	Let $f,g\in U(1)^\ast$, and $a,b\in U(1)$. Since $L(-1)a_1 f\in U(1)^\ast$, and $U(1)$ consists of quasi-primary vectors, by \meqref{2.7} we have $\<L(-1)(a_1 f),b\>=\<a_1 f, L(1)b\>=0.$ Moreover, since $a_j f\in U(1-j)^\ast=0$ for $j\geq 2$, by \meqref{3.7} and \meqref{3.29} we derive
	\begin{align*}
	\<f(0)a,b\>&=\Res_{z} \<e^{zL(-1)}Y_{U'}(a,-z)f,b\>=\sum_{j\geq 0} \<(-1)^j\frac{L(-1)^j}{j!} (a_j f),b\>\\
	&=\<a_0 f,b\>-\<L(-1)(a_1 f),b\>
=\<\rho^\ast(a)(f),b\>.
	\end{align*}
	Hence $f(0)a=\rho^\ast(a)(f)$. 	If $T_r: U'\ra U$ is a $0$-relative RBO, then by \meqref{3.28b} and the definition formulas \meqref{3.8} and \meqref{4.60}, we have
	\begin{align*}
	0&=T_r(f)_0 T_r(g)-T_r(T_r(f)_0 g)-T_r(f(0)T_r(g))\\
	 &=[T_R(f),T_R(g)]-T_R(\rho^\ast(T_R(f))(g))-T_R(\rho^\ast(T_R(g))(f)),
	\end{align*}
	for all $f,g\in U(1)$. Thus $T_R:U(1)^\ast\ra U(1)$ is a relative RBO.
	\end{proof}
With the notations in this subsection, by Lemma \mref{lm4.1}, Lemma \mref{lm4.2} and Theorem~\mref{thm:main1}, we have the following diagram. 
\vspace{-.1cm}
\begin{equation}\label{eq:diag1}
\begin{tikzcd}
\text{s.-s. } r\text{ is solution to 0-VOYBE in } U \arrow[d,Leftrightarrow,"\text{Thm. \ref{thm:main1}}"'] \arrow[r, "\text{Lem.~\mref{lm4.1}}"] &  \text{s.-s. } R \text{ is solution to the CYBE in } U(1) \arrow [d,Leftrightarrow,"\text{\mcite{K1}}"]\\
T_r \text{ is  a 0-relative RBO of } U  \arrow[r, "\text{Lem.~\mref{lm4.2}}"]& T_R \text{ is a relative RBO of } U(1)
\end{tikzcd}
\end{equation}
where ``s.-s.'' in this diagram is the abbreviation of ``skewsymmetric''.
In particular, the classical result about constructing a relative RBO from a solution of the CYBE~\mcite{K1} can be viewed as a corollary of Theorem~\mref{thm:main1}.

\subsection{Solving the CYBE from relative RBOs}
We can also recover the process of using relative RBOs to produce solutions of the CYBE in the semidirect product Lie algebras~\mcite{B2} by restricting the corresponding process for the VOYBE to the first-levels.

In this subsection, we let $(V,Y,\vac,\om)$ be a VOA, and $(W,Y_W)$ be a $V$-module of conformal weight $0$. Let $U=V\rtimes W'$ as in Section \mref{ss:3.1}. Let $T\in \Hom_{\mathrm{LP}}(W,V)$. For $r=T-\sigma(T)=\sum_{t=0}^\infty r^t= \sum_{t=0}^{\infty}\sum_{i=1}^{p_t} T(v_i^t)\otimes (v_i^t)^\ast- (v_i^t)^\ast\otimes T(v_i^t)\in U^{\mhat{\o}2}$ in \meqref{3.13}, we let
\begin{equation}\mlabel{3.63}
\mathcal{R}:=r^1=\sum_{i=1}^{p_{1}} T(v_i^1)\otimes (v_i^1)^\ast- (v_i^1)^\ast\otimes T(v_i^1)\in (V_1\op W(1)^\ast)\otimes (V_1\op W(1)^\ast)
\end{equation}
be the homogeneous part of $r$ in $U(1)\o U(1)$, where $\{v^1_{1},\dots , v^1_{p_1}\}$ is a basis of $W(1)$ while $\{(v^1_1)^\ast,\dots, (v^1_{p_1})^\ast\}$ the dual basis of $W(1)^\ast$.

\begin{lm}\mlabel{lm4.3}
	Assume that $V_1$, $W(1)$ and $W(1)^\ast$ are spanned by quasi-primary vectors. Then the operations $\cdot_{0},\cdot'_0$, and $\cdot'^{\mathrm{op}}_0$ in Definition \mref{df2.3} satisfy
\begin{align*}
&a_0b=[a,b], && a_0'b=-[a,b],&& a_0'^{\mathrm{op}} b=[a,b];\\
&a_0v^\ast=\rho^\ast(a) v^\ast,&& a_0'v^\ast=-\rho^\ast(a)v^\ast,&& a'^{\mathrm{op}}_0 v^\ast=\rho^\ast(a)v^\ast;\\
&v^\ast(0) a=-\rho^\ast(a)v^\ast,&& (v^\ast)'(0)a=\rho^\ast(a)v^\ast,&& (v^\ast)'^{\mathrm{op}}(0)a=-\rho^\ast(a)v^\ast,
\end{align*}
for $a,b\in V_1$ and $v^\ast\in W(1)^\ast$, where $\rho^\ast$ is given by \meqref{3.29}.
\end{lm}
\begin{proof}
Since $\wt (a)=\wt (b)=1$, and $L(1)a=L(1)b=L(1)v^\ast=0$, the equations on first two rows follow immediately from \meqref{3.17}, \meqref{3.18} and \meqref{3.29}. Let $u\in W(1)$. By assumption we have $L(1)u=0$, and $\wt (u)=\wt (v^\ast)=1$. Then
	\begin{align*}
\<v^\ast(0)a,u\>&=\Res_z \<Y_{W'V}^{W'}(v^\ast, z)a,u\>=\Res_z\<Y_{W'}(a,-z)v^\ast,e^{zL(1)}u\>
=\<-a_0v^\ast, u\>=\<-\rho^\ast(a)v^\ast, u\>,\\
\<(v^\ast)'(0)a, u\>&=\Res_z \<Y_{W'V}^{W'}(e^{zL(1)}(-z^{-2})v^\ast,z^{-1})a,u\>=\Res_z(-1)z^{-2}\<Y_{W'}(a,-z^{-1})v^\ast,e^{z^{-1}L(1)}u\>\\
&=\Res_z (-1) z^{-2} \<\sum_{n\in \Z}a_nv^\ast (-1)^{n+1}z^{n+1},u \>=\<\rho^\ast(a)v^\ast,u\>,\\
\<(v^\ast)'^{\mathrm{op}}(0)a,u\>&=\Res_z \<Y_{W'V}^{W'}(e^{-zL(1)}(-z^{-2})^{L(0)}v^\ast, -z^{-1})e^{zL(1)}a,u\>\\
&=\Res_z (-1)z^{-2}\<Y_{W'}(a,z^{-1})v^\ast,e^{z^{-1}L(1)}u\>\\
&=\Res_z(-1)z^{-2}\<\sum_{n\in \Z} a_n v^\ast z^{n+1},u \>=\<-\rho^\ast(a)v^\ast, u\>.
\end{align*}
Since $v^\ast(0) a, (v^\ast)'(0)a$ and $(v^\ast)'^{\mathrm{op}}(0)a$ are contained in $W(1)^\ast$, we arrive at the conclusion.
\end{proof}

Recall (cf. \mcite{B2}) that $V_1\op W(1)^\ast$ carries a semidirect product Lie algebra structure:
\begin{equation}\mlabel{3.64}
[a+u^\ast, b+v^\ast]=[a,b]+\rho^\ast(a) v^\ast-\rho^\ast(b) u^\ast,\quad  a,b\in V_1,\ u^\ast,v^\ast\in W(1)^\ast.
\end{equation}
\begin{prop}\mlabel{prop4.4}
	Assume that $V_1$, $W(1)$ and $W(1)^\ast$ are spanned by quasi-primary vectors.
	Let $r=T-\sigma(T) \in U^{\mhat{\o}2}$ be as in \meqref{3.13}, and let $\R=r^1$ be as in \meqref{3.63}. If $r$ is a skewsymmetric solution to the $0$-VOYBE in $V\rtimes W'$ $($see Definition \mref{df:VOYBE}$)$, then $\R \in (V_1\rtimes W(1)^\ast)^{\otimes 2}$ is a skewsymmetric solution to the classical Yang-Baxter equation in the Lie algebra $V_1\rtimes W(1)^\ast$
	\begin{equation}
	 [\R_{12},\R_{13}]+[\R_{12},\R_{23}]+[\R_{13},\R_{23}]=0.
	\end{equation}
\end{prop}
\begin{proof}
	Define a projection map on $U^{\mhat{\o}3}$ by
\vspace{-.2cm}
	$$p_{1,1,1}: U\mhat{\o}U\mhat{\o}U\ra U(1)\otimes U(1)\otimes U(1), \ \ \  p_{1,1,1}\Big(\sum_{q,s,t=0}^\infty \sum_{i,j,k}\al_i^q\otimes \b_j^{s}\otimes \ga_k^{t}\Big)= \sum_{i,j,k}\al_i^{1}\otimes \b_j^{1}\otimes \ga_k^{1},$$
	where the sums over $i,j,k$ are finite, and $\al^q_i\in U(r)$, $\b^s_j\in U(s)$, and $\ga^t_k\in U(t)$, for $q,s,t\geq 0$, and $i,j,k\geq 1$. By \meqref{3.15}-\meqref{3.17}, the computations in the proof of Theorem \mref{thm:main2}, together with Lemma \mref{lm4.3}, we derive
{\small 
	\begin{align*}
&p_{1,1,1}(r_{12}\cdot_{0}r_{13})=r^1_{12}\cdot_0 r^1_{13}\\
	&= \sum_{i,k} \bigg(T(v^1_k)_0 T(v^1_i)\otimes (v^1_i)^\ast\otimes (v^1_k)^\ast- T(v^1_k)_0(v^1_i)^\ast \otimes T(v^1_i)\otimes (v^1_k)^\ast
	 -(v^1_k)^\ast(0)T(v^1_i)\otimes (v^1_i)^\ast\otimes T(v^1_k)\bigg)\\
	&=\sum_{i,k} \bigg(-[T(v^1_i), T(v^1_k)]\otimes (v^1_i)^\ast\otimes (v^1_k)^\ast- \rho^\ast (T(v^1_k))(v^1_i)^\ast\otimes T(v^1_i)\otimes (v^1_k)^\ast\\
	&\quad +\rho^\ast(T(v^1_i))(v^1_k)^\ast\otimes (v^1_i)^\ast\otimes T(v^1_k)\bigg)=-[\R_{12},\R_{13}],\\
&	p_{1,1,1}(-r_{23}\cdot'_0 r_{12})= -r^1_{23}\cdot'_0 r_{12}\\
&=\sum_{i,l} \bigg(-T(v^1_i)\o T(v^1_l)'_0 (v^1_i)^\ast\o (v^1_l)^\ast+ (v^1_i)^\ast\o T(v^1_l)'_0 T(v^1_i)\o (v^1_l)^\ast
	-(v^1_l)^\ast \o ((v^1_l)^\ast)'(0) T(v^1_i)\o T(v^1_l)	\bigg)\\
	&=\sum_{i,l} \bigg(T(v^1_i)\o \rho^\ast(T(v^1_l))(v^1_i)^\ast\o (v^1_l)^\ast+ (v^1_i)^\ast\o [T(v^1_i),T(v^1_l)]\o (v^1_l)^\ast\\
	&\quad -(v^1_i)^\ast \o \rho^\ast(T(v^1_i))(v^1_l)^\ast\o T(v^1_l)\bigg)=-[\R_{12},\R_{23}],\\
&	p_{1,1,1}(r_{13}\cdot'^{\mathrm{op}}_0 r_{23})=r^1_{13}\cdot'^{\mathrm{op}}_0 r_{23}\\
&=\sum_{k,l} \bigg(-T(v^1_k)\o (v^1_l)^\ast\o T(v^1_l)'^{\mathrm{op}}_0(v^1_k)^\ast-(v^1_k)^\ast\o T(v^1_l)\o ((v^1_l)^\ast)'^{\mathrm{op}}(0) T(v^1_k)\\
&\quad+ (v^1_k)^\ast\o (v^1_l)^\ast\o T(v^1_l)'^{\mathrm{op}}_0 T(v^1_k) \bigg)\\
&=\sum_{k,l}\bigg(-T(v^1_k)\o (v^1_l)^\ast\o \rho^\ast(T(v^1_l))(v^1_k)^\ast+ (v^1_k)^\ast\o T(v^1_l)\o \rho^\ast(T(v^1_k))(v^1_l)^\ast\\
&\quad - (v^1_k)^\ast\o (v^1_l)^\ast\o [T(v^1_k),T(v^1_l)] \bigg)=-[\R_{13},\R_{23}].
	\end{align*}
}
 Since the projection $p_{1,1,1}$ is clearly a linear map, we have
$$
	0=p_{1,1,1}(r_{12}\cdot_0 r_{13}-r_{23}\cdot_0 r_{12}+r_{13}\cdot_0 r_{23})=-[\R_{12},\R_{13}]-[\R_{12},\R_{23}]-[\R_{13},\R_{23}].
$$
	Hence $\R\in (V_1\rtimes W(1)^\ast)^{\otimes 2}$ is a solution to the classical Yang-Baxter equation.
\end{proof}
On the other hand, consider the restriction of the level-preserving linear map $T\in \Hom_{\mathrm{LP}}(W,V)$ to the first level. We write
\begin{equation}
\T:=T|_{W(1)}: W(1)\ra V_1.
\end{equation}

\begin{prop}\mlabel{prop4.5}
	Let $V_1$ and $W(1)$ be spanned by quasi-primary vectors, and let $T\in \Hom_{\mathrm{LP}}(W,V)$. If either one of the conditions \meqref{3.8}, \meqref{3.39}, or \meqref{3.40} with $m=0$ holds, then $\T:W(1)\ra V_1$ is a relative RBO. In particular, if $T$ is a $0$-strong relative RBO, then $\T$ is a relative RBO.
\end{prop}
\begin{proof}
Let $u,v\in W(1)$ and $f\in V_1^\ast$. Assume that $T$ satisfies \meqref{3.8}. Then we have
\begin{align*}
[\T(u),\T(v)]&=\T(u)_0 \T(v)=\T(\T(u)_0v)+\T(u(0)\T(v))\\
&=\T(\rho(\T(u))v)+ \T\Big(\sum_{j\geq 0} \frac{(-1)^{j+1}}{j!} L(-1)^j (\T(v))_j u\Big)\\
&=\T(\rho(\T(u))v)-\T(\rho(\T(v))u)+\T(L(-1)\T(v)_1 u).
\end{align*}
Note that $L(-1)\T(v)_1u\in W(1)$, and $\<L(-1)\T(v)_1u, w\>=\<\T(v)_1u,L(1)w\>=0$, for all $w\in W(1)$. Hence $L(-1)\T(v)_1u=0$. Then we have $\T(u(0)\T(v))=-\T(\rho(\T(v))u)$ and $[\T(u),\T(v)]=\T(\rho(\T(u))v)-\T(\rho(\T(v))u)$. Thus $\T$ is a relative RBO.

Assume $T$ satisfies \meqref{3.39}. Note that $\T^\ast=T^\ast|_{V^\ast_{1}}: V_1^\ast\ra W(1)^\ast$. Then by \meqref{3.37} and the definition of $Y_{VV'}^{V'}$ and $Y_{W'}$ in Section 5.4 in \mcite{FHL}, we obtain
\begin{align*}
0&=	 \Res_{z}\left\<Y_{W'}(T(u),z)T^\ast(f)-T^\ast(Y_{VV'}^{V'}(T(u),z)f)+T^\ast(\Y_{WW'}^{V'}(u,z)T^\ast(f)) , v\right\>\\
&=\Res_z\left<\T^\ast(f),  Y_{W}(e^{zL(1)}(-z^{-2})^{L(0)}\T(u), z^{-1})v \right\>-\Res_z \left\<f,Y_{V}(e^{zL(1)}(-z^{-2})^{L(0)}\T(u),z^{-1})\T(v) \right\>\\
&\quad +\Res_z\left\<\T^\ast(f), Y_{WV}^W(e^{zL(1)}(-z^{-2})^{L(0)}u,z^{-1})\T(v) \right\>\\
&=-\<f,\T\left(\T(u)_0v\right)\>+\<f,\T(u)_0\T(v)\>-\<f,\T\left(u(0)\T(v)\right)\>\\
&=\<f,[\T(u),\T(v)]-\T(\rho(\T(u))v)+\T(\rho(\T(v))u)\>.
\end{align*}
Hence $\T$ is a relative RBO. Finally, assume that $T$ satisfies \meqref{3.40}. Then
\begin{align*}
0&= \Res_{z} \left\<Y_{W'V}^{W'}(T^\ast(f),z)T(u)-T^\ast(Y_{V'V}^{V'}(f,z)T(u))+T^\ast(\Y_{W'W}^{V'}(T^\ast(f),z)u),v\right\>\\
&=\Res_z \left\<Y_{W'}(\T(u),-z)\T^\ast(f),v\right\>-\<Y_{V'}(\T(u),-z)f,\T(v)\>+\<\Y_{WW'}^{V'}(u,-z)\T^\ast(f),\T(v)\>\\
&=\<f,\T\left(\T(u)_0v\right)\>-\<f,\T(u)_0\T(v)\>+\<f,\T\left(u(0)\T(v)\right)\>,
\end{align*}
and so $\T$ is a relative RBO.
\end{proof}
The proof of Proposition \mref{prop4.5} immediately gives an easy way to construct $0$-strong relative RBOs like in Example \mref{ex2.10}.
\begin{coro}
	Let $V_1$ and $W(1)$ be spanned by quasi-primary vectors, and let $\phi: W(0)\ra V_0=\C\vac$ be an arbitrary linear map. Then a relative RBO $\T:W(1)\ra V_1$ of the Lie algebra $V_1$ can be extended to a $0$-strong relative RBO $T:W\ra V$ by letting
	$$T|_{W(0)}:=\phi,\quad T|_{W(1)}:=\T,\ \ \mathrm{and}\ \ T|_{W(n)}:=0,\quad n\geq 2.$$
	\end{coro}

Now apply Propositions \mref{prop4.4} and \mref{prop4.5} and assume that $V_1$, $W(1)$ and $W(1)^\ast$ are spanned by quasi-primary vectors. Then we have another diagram that illustrate the relationship between the $0$-VOYBE and $0$-strong relative RBO of VOAs on the one hand, and the CYBE and the relative RBO of Lie algebras on the other.
\begin{equation}\mlabel{eq:diag2}
\begin{tikzcd}
T\text{\ is\ a\ 0-strong\ relative\ RBO\ of VOA} \arrow[d,Leftrightarrow,"\text{Thm.\ \mref{thm:main2}} "'] \arrow[r, "\text{Prop.~\mref{prop4.5}}"] & \T\text{\ is a relative RBO of Lie algebra} \arrow [d,Leftrightarrow,"\text{\mcite{B2}}"]\\
r=T-T^{21}\text{\ is \ a\ solution\ to\ 0-VOYBE\ } \arrow[r, "\text{Prop.~\mref{prop4.4}}"]& \R=\T-\T^{21} \text{\ is\ a\ solution\ to\ CYBE}
\end{tikzcd}
\end{equation}

\noindent
{\bf Acknowledgments.} This research is supported by
NSFC (11931009, 12271265, 12261131498), the Fundamental Research Funds for the Central Universities and Nankai Zhide Foundation.

\noindent
{\bf Declaration of interests. } The authors have no conflicts of interest to disclose.

\noindent
{\bf Data availability. } No new data were created or analyzed in this study.

\end{document}